\documentclass[hidelinks,12pt]{article}
\usepackage{amsfonts,amsmath,amssymb,amsthm}
\usepackage{mathtools}
\usepackage{graphicx}
\usepackage{hyperref,cleveref}
\usepackage{microtype}
\usepackage{verbatim}
\usepackage{tikz-cd}

\newcommand{\C}{\mathbb{C}}
\newcommand{\R}{\mathbb{R}}
\newcommand{\Z}{\mathbb{Z}}

\newtheorem{theorem}{Theorem}[section]
\newtheorem{lemma}[theorem]{Lemma}

\newtheorem{definition}[theorem]{Definition}
\newtheorem{proposition}[theorem]{Proposition}
\newtheorem{corollary}[theorem]{Corollary}
\newtheorem{conjecture}[theorem]{Conjecture}

\newtheoremstyle{BoldRemark} 
{\topsep}                    
{\topsep}                    
{\upshape}                   
{}                           
{\bfseries}                  
{.}                          
{.5em}                       
{}  
\theoremstyle{BoldRemark}
\newtheorem{remark}[theorem]{Remark}

\newtheorem{construction}[theorem]{Construction}

\title{Real Global Group Laws and Hu-Kriz Maps}
\author{Jack Carlisle, Noah Wisdom, Guoqi Yan}
\date{}

\begin{document}

\maketitle
\begin{abstract}
        Recently, Hausmann defined global group laws and used them to prove that $MU^G_*$ is the $G$-equivariant Lazard ring, for $G$ a compact abelian Lie group. On the other hand, Hu and Kriz showed that the restriction map induces an isomorphism $M \R^{C_2}_{\rho *} \cong MU_{2*}$. In this paper, we blend these stories. We utilize the $C_2$-global spectrum $\mathbf{MR}$ defined by Schwede in an unpublished note, which gives rise to a genuine $G$-spectrum $M \mathbb{R}_\eta$ for each augmented compact Lie groups $\eta: G\to C_2$, simultaneously generalizing $MU_G$ and $M\R$. In the case of semi-direct product augmentations $G \rtimes C_2\to C_2$ with $G$ compact abelian Lie and $C_2$ acting by inversion, we show that the restriction along the inclusion $G \subset G \rtimes C_2$ is a split surjection $M \R^{G \rtimes C_2}_{\rho *} \rightarrow MU^{G}_{2*}$. Additionally, we propose an evenness conjecture, which implies that this map is an isomorphism. Along the way, we define Real $\eta$-orientations, Real global orientations, and corresponding notions of equivariant and global group laws.
\end{abstract}
\tableofcontents

\section{Introduction}

The complex cobordism spectrum $MU$ is perhaps the most important object in chromatic homotopy theory. Letting $C_2$ act by complex conjugation gives rise to a (genuine) $C_2$ spectrum $M \mathbb{R}$ which is important in equivariant chromatic homotopy theory (roughly speaking, this is the program of studying classical chromatic homotopy theory through natural group actions on classically important chromatic spectra, such as the Morava stabilizer group action on Lubin-Tate theory and the work of Hill-Hopkins-Ravenel \cite{HHR}, and includes the theory of odd primary orientations recorded by \cite{HSW20}). Additionally, $M \mathbb{R}$ carries the universal Real orientation \cite{Ara79} and is related to the cobordism theory of Real manifolds \cite{Hu99} \cite{HK01}.

On the other hand, there is a different genuine $C_2$-spectrum, $MU_{C_2}$, important in chromatic equivariant homotopy theory, which is the program of attempting to mimic the classical chromatic story in the genuine equivariant setting. In fact, for every compact Lie group $G$, there is a genuine $G$-spectrum $MU_G$ whose underlying spectrum is $MU$. This spectrum is universal one with Thom isomorphisms for $G$-equivariant vector bundles \cite{Oko82} \cite{Cos87}. Additionally, Hausmann has shown that, for $G$ abelian, $MU_G^*$ carries the universal $G$-formal group law, and Hausmann and Meier have shown that the spectrum of invariant prime ideals of $MU_G^*$ is naturally isomorphic to the Balmer spectrum of the category of genuine $G$-spectra \cite{Hau22} \cite{HM23}.

The present article began as a way to quantify the difference between $M \mathbb{R}$ and $MU_{C_2}$. It turns out that there is actually a genuine $C_2 \times C_2$-spectrum $M \mathbb{R}_{C_2}$ for which restriction to one copy of $C_2$ yields $M \mathbb{R}$ and restriction to the other copy yields $MU_{C_2}$. More generally, for any compact Lie group morphism $\eta : G \rightarrow C_2$, there is a genuine $G$-spectrum $M \mathbb{R}_\eta$, such that restriction to $Ker(\eta) \subset G$ yields $MU_{Ker(\eta)}$, and, if $\eta$ admits a section, restriction to the image of that section yields $M \mathbb{R}$. We refer to the data of such an $\eta$ as an \emph{augmented compact Lie group}. When $\eta$ is the trivial homomorphism, we have $M \mathbb{R}_\eta \cong MU_G$.

Hu and Kriz show in \cite{HK01} that restriction induces an isomorphism \[ \pi^{C_2}_{\rho *}(M \mathbb{R}) \cong \pi_{2*}(MU) \] of graded rings, where $\rho$ denotes the representation of $C_2$ on $\mathbb{C}$ by complex conjugation. The main result of our paper is progress towards a generalization of this. Before we can state it, we must give some definitions.

An \emph{augmented torus} $\widehat{\mathbb{T}}^n$ is the augmentation $\mathbb{T}^n \rtimes C_2 \rightarrow C_2$ given by projection, where $C_2$ acts on $\mathbb{T}^n$ via inversion. By abuse, we may use $\widehat{\mathbb{T}}^n$ to denote both the augmentation map and the group $\mathbb{T}^n \rtimes C_2$. This construction actually makes sense when the torus is replaced with any compact abelian Lie group, and the class of all augmented groups obtained this way, along with the trivially augmented compact abelian Lie groups, are the so-called \emph{quasi-abelian} augmented Lie groups.

Given an augmented group $\eta : G \rightarrow C_2$, we obtain a map $\eta^* : RO(C_2) \rightarrow RO(G)$. We often omit the $\eta^*$ from the notation, so that if we write down the $C_2$-representation $\sigma$ or $\rho$ (respectively the sign representation over $\R$ and the regular representation over $\R$), we may implicitly view these as $G$-representations over $\R$.

\begin{theorem}
\label{thm:main-thm-first-form}
	Let $G$ be a compact abelian Lie group and $\hat{G}=G\rtimes C_2$ with $C_2$ acting by inversion. Then restriction to $G \subset \hat{G}$ defines a split surjection \[ \pi^{\hat{G}}_{\rho *} ( M \mathbb{R}_{\hat{G}}) \twoheadrightarrow \pi^{G}_{2*}(MU_G) \textrm{.} \]
\end{theorem}

The idea is to show that $\pi^{\hat{G}}_{*\rho}M\R_{\hat{G}}$ carries a $G$-equivariant formal group law that restricts to the universal one on $\pi^G_{2*}MU_G$. This $G$-equivariant formal group law will arise from a so-called Real $G$-equivariant formal group law which in turn is built using a notion of Real $G$-equivariant orientation. We will show that $M \mathbb{R}_\eta$ is appropriately Real $G$-equivariantly oriented by appealing to global ideas, particularly those developed in the celebrated paper of Hausmann \cite{Hau22}. 

Global homotopy was developed in \cite{Boh12} and \cite{Sch18}, and heuristically allows one to work $G$-equivariantly for all $G$ one cares about simultaneously. We begin by recalling the Real global homotopy first developed by Schwede in an unpublished early draft \cite{Sch14} or \cite{Sch18}, with many details filled in by \cite[Appendix A]{Sch22} (where it is referred to as $C_2$-global homotopy theory). Heuristically, Real global homotopy theory allows one to work with all augmented groups simultaneously. Next, we define the notion of a \emph{Real global orientation}, the Real analogue of Hausmann's notion of complex oriented global spectrum \cite{Hau22}, and observe that any Real globally oriented spectrum gives rise to both a collection of Real $G$-equivariant oriented spectra, and an algebraic gadget which we call a \emph{Real global group law}. This is similar to the notion of a global group law, except that it is built out of $RO(C_2)$-graded rings; for a precise definition, see Definition \ref{def:Real-GGL}.

\begin{theorem}
    Let $\mathbf{E}$ be a Real globally oriented Real global spectrum. Then the $RO(C_2)$-graded homotopy groups \[ \pi^{\widehat{\mathbb{T}}^n}_{\bigstar} \mathbf{E} \] form a Real global group law.
\end{theorem}

We will construct a Real globally oriented spectrum $\mathbf{MR}$, based on a construction in \cite{Sch14}, and consequently obtain a Real global group law. There is a functor from the category of Real global group laws to global group laws, and the fact that $\pi^{\mathbb{T}^n}_{2*} MU_{\mathbb{T}^n}$ form the initial global group law gives another way to prove Theorem \ref{thm:main-thm-first-form} for augmented tori.

To show that the split surjection of Theorem \ref{thm:main-thm-first-form} is actually an isomorphism, we propose an evenness conjecture, which the second and third authors plan to prove in future work. This conjecture yields an immediate proof that the map of Theorem \ref{thm:main-thm-first-form} is injective, hence an isomorphism.

\begin{conjecture}[Evenness Conjecture, Conjecture \ref{conj:MR-eta-is-even}]\label{conj: evennes conjecture in introduction}
    For all quasi-abelian augmented Lie groups $\eta : G \rightarrow C_2$, we have \[ \pi^G_{* \rho - 1}( M \mathbb{R}_\eta ) = 0 \textrm{.} \]
\end{conjecture}

On the other hand, there is a consequence of the Evenness Conjecture (cf Proposition \ref{prop: Evenness Conj implies Regularity Conj}) which allows one to deduce that the map of Theorem \ref{thm:main-thm-first-form} is injective merely for the augmented tori.

\begin{conjecture}[Regularity conjecture, Conjecture \ref{conj:R-is-regular}]\label{conj:regularity conjecture in introduction}
    The global group law $\mathbf{R}$ is regular in the sense of \cite[Definition 5.9]{Hau22}. In particular, the maps $\mathbf{R}(\mathbb{T}^n) \rightarrow \Phi^{\mathbb{T}^n} \mathbf{R}$ are injective.
\end{conjecture}

Even thought it follows from the Evenness Conjecture, we still call the above statement the Regularity Conjecture, since it is strictly weaker than the Evenness Conjecture and it is still strong enough to prove the restriction isomorphism for augmented tori (in fact, the Regularity Conjecture is equivalent to a version of the Evenness Conjecture which only holds for augmented tori). Assuming the Regularity Conjecture, we may mimic the arguments of \cite{Hau22} to reduce the proof of injectivity to a statement about geometric fixed points. We give a self-contained proof, which does not rely on any conjectural assumptions, of the following result.

\begin{theorem}
\label{thm:res-induces-iso-on-geom-fixed-points}
    Let $\hat{G} = G \rtimes C_2$ for $G$ any compact abelian Lie group. Then the restriction map $\Phi^G(res^{\widehat{G}}_{G})$ given by 
    \begin{align*} 
    \pi_{*\rho}^{C_2}\Phi^G & (M\R_{\hat{G}}) \cong [S^{*\rho},M\R_{\hat{G}}\wedge \widetilde{E\mathcal{F}[G]}]^{\hat{G}} \\
    & \to [S^{2*},MU_G\wedge \widetilde{E\mathcal{P}}]^G\cong \pi_{2*}\Phi^G MU_G
    \end{align*}
    is an isomorphism.
\end{theorem}

With the assumption of either conjecture, we offer the following slicker statement of our main theorem. We use $\mathbf{L}$ to denote the global group law determined by $\mathbf{MU}$, and $\mathbf{R}$ to denote the global group law obtained from applying a certain functor to the Real global group law determined by $\mathbf{MR}$.

\begin{theorem}
\label{thm:main-theorem-second-form}
	Assume either Conjecture \ref{conj:MR-eta-is-even} or Conjecture \ref{conj:R-is-regular}. Then the canonical map $\mathbf{L} \rightarrow \mathbf{R}$ is an isomorphism with inverse given levelwise by restriction $\mathbf{R} \rightarrow \mathbf{L}$.
\end{theorem}

In fact, if one assumes the stronger Evenness Conjecture, one may further strengthen this result. We thank Markus Hausmann for pointing out that our proof for the augmented-tori case should imply the case of all quasi-abelians under these assumptions.

\begin{theorem}
    Assume Conjecture \ref{conj:MR-eta-is-even}. Then for any quasi-abelian augmented Lie group $\eta : \hat{G} \rightarrow C_2$, the restriction map \[ \pi^{\hat{G}}_{\rho *}( M \mathbb{R}_{\hat{G}} ) \rightarrow \pi^{ker(\eta)}_{2 *} (MU_{ker(\eta)}) \] is an isomorphism.
\end{theorem}

Most of the material in this paper can be read model independently. The only place where we choose orthogonal $G$-spectra developed in \cite{Sch22} for our model of genuine $G$-spectra is to get a commutative ring structure on $M \mathbb{R}_{\eta}$ (on the point-set level) using that of $M \mathbb{R}$.

We begin in section 2 with a review of the representation theory of augmented Lie groups and Real global homotopy theory, following the exposition in unpublished work of Schwede \cite{Sch14} and the rigorous treatment in \cite{Sch22}. Next, in section 3 we describe analogues of geometric constructions from \cite{Hau22} for augmented Lie groups which we require, and in section 4 we apply these geometric constructions to analyze Real orientations. In section 5 we give Schwede's construction of the Real global spectrum $\mathbf{MR}$, which by the machinary of section 4 gives rise to a global group law $\mathbf{R}$. After this, we have the necessary ingredients to prove our main theorems in section 6.

\subsection{Notation}

\begin{itemize}
    \item $\eta : G \rightarrow C_2$ will denote an augmented compact Lie group.
    \item $\widehat{\mathbb{T}}^n$ will denote the augmentation $\mathbb{T}^n \rtimes C_2 \rightarrow C_2$ given by projection, where $C_2$ acts on $\mathbb{T}^n$ via inversion. By abuse, we may use $\widehat{\mathbb{T}}^n$ to denote both the augmentation map and the group $\mathbb{T}^n \rtimes C_2$.
    \item If $G$ is a compact abelian Lie group, we will use $\hat{G}$ to denote $G\rtimes C_2$ with $C_2$ acting by group inverson. This is the only semi-direct product we use in this paper.
    \item $G^*$ is the character group $Hom(G,S^1)$ of a compact Lie group $G$.
\end{itemize}

\subsection{Acknowledgements}

The authors thank Ryan Quinn for suggesting the definition of Real global group laws, along with the conjectures on the vanishing of the $* \rho-1$ homotopy groups, and of the regularity of the global group law $\mathbf{R}$. The authors would also like to thank Yutao Liu for many helpful and productive conversations, and Stefan Schwede for sharing an early draft of his book \cite{Sch14} which contained much of the necessary technical background. The second author thanks Mike Hill for an initial suggestion which lead to this project. Also, the authors thank Lucas Piessevaux, Tobias Lenz, Lennart Meier, Markus Hausmann, and Robert Burklund for valuable insights. Finally, the authors thank Markus Hausmann for providing very helpful feedback on an early draft.

The second author was supported by National Science Foundation grant number DMS-2152235 and the Max Plank Institute for Mathematics at Bonn.

\section{Real global homotopy theory}

One of the main features of equivariant homotopy theory is that $\Z$-graded objects often acquire $RO(G)$-gradings. In our setting, $G$ will be replaced by an \emph{augmented goup}, that is, a continuous homomorphism $G \rightarrow C_2$. We will therefore require an understanding of the representation theory of such objects.

\subsection{Augmented Lie groups and their representation theory}

Here we introduce augmented compact Lie groups. In the present work, augmented compact Lie groups and families of such play a role analogous to that of compact Lie groups and families of such in global homotopy theory, as utilized in \cite{Hau22}. We also record some basic representation-theoretic facts in this setting. The material in this appendix is loosely adapted from an early draft of Schwede's text \cite{Sch14}.

\begin{definition}
    An augmented Lie group is a continuous homomorphism $\eta : G \rightarrow C_2$ from a compact Lie group $G$ to the order two cyclic group $C_2$.
\end{definition}

The collection of such forms a category, the slice category of compact Lie groups over $C_2$. Morphisms are maps which respect the augmentation; they are given by commutative squares \[ \begin{tikzcd}
    G \arrow{d} \arrow{r} & H \arrow{d} \\
    C_2 \arrow{r}{Id} & C_2 
\end{tikzcd} \] This category has finite products: given $G \rightarrow C_2$ and $H \rightarrow C_2$, the categorical product is the fiber product $G \times_{C_2} H$, equipped with the canonical map to $C_2$.

Given an augmented group $G$, we can extract as a subgroup the kernel of the augmentation map. This construction specifies a functor from augmented groups to groups. On the other hand, there are multiple ways to make a compact Lie group into an augmented group. Given a compact Lie group $G$, we could either give $G$ the trivial augmentation, or form the product $G \times C_2$, with the projection map as the augmentation. The underlying compact Lie group of either of these constructions returns $G$, and both constructions are functorial. 

Alternatively, the most important example of producing an augmented group from an ordinary group requires us to assume commutativity. Specifically, given a compact abelian Lie group $A$, we can form the semidirect product $G := A \rtimes C_2$, with $C_2$ acting on $G$ by inversion: $\gamma x = -x$. Of central importance will be the augmented tori $\widehat{\mathbb{T}}^n$ obtained by feeding in an $n$-dimensional torus to the above construction. By abust we will use $\widehat{\mathbb{T}}^n$ to refer to both the augmentation map and the domain of the augmentation.

A particularly important class of augmented groups are constructed as follows. For $V$ a complex inner product space, define $\mathbf{L}^\C(V,V)$ as the collection of maps $V \rightarrow V$ which are either $\C$-linear isometries or conjugate linear maps which preserve the inner product up to conjugation. The augmentation sends $\C$-linear maps to the trivial element in $C_2$ and conjugate linear maps to the nontrivial element in $C_2$.

\begin{definition}
    A Real representation of an augmented Lie group $\eta : G \rightarrow C_2$ is an augmented map $G \rightarrow \mathbf{L}^\C(V,V)$. We may sometimes refer to a Real representation of $\eta$ as just a representation of $\eta$.
\end{definition}

Unraveling the definitions, we see that a Real representation of $\eta$ is an $\R$-linear representation of $G$ for which each element of $G$ acts either by a $\C$-linear isometry or a conjugate linear anti-isometry depending on the augmentation. A morphism of representations is an isometric embedding $V \rightarrow W$ which is $G$-equivariant.

Given two Real $\eta$-representations, both the direct sum and the tensor product become $\eta$-representations. In particular, we may define an irreducible Real $\eta$-representation as one which does not admit any nontrivial direct sum decompositions.

\begin{proposition}[\cite{Sch14}]\label{prop:quasi-abelian}
The following are equivalent for an augmented Lie group $\eta$, in which case we refer to $\eta$ as quasi-abelian.
\begin{enumerate}
	\item Every irreducible Real $\eta$-representation has $\C$-dimension $1$.
	\item $G$ embeds as an augmented closed subgroup of $\widehat{\mathbb{T}^n} := \mathbb{T}^n \rtimes C_2$ for some torus $\mathbb{T}^n$.
	\item The kernel of the augmentation $Ker(\eta)$ is abelian and every element of $G$ which augments to the nontrivial element of $C_2$ has order $2$.
	\item $G \cong A \rtimes C_2$ for some abelian $A$ with $C_2$ acting by inversion, or $G$ is abelian and has the trivial augmentation.
\end{enumerate}
\end{proposition}

\begin{proof}
(1) $\implies$ (2) First note that there is a faithful Real $\eta$-representation (in particular if $V$ is a faithful $Ker(\eta)$-representation over $\R$ then a modification of $Ind_{Ker(\eta)}^{G} V$ may be made to work). This representation is a direct sum of one-dimensional representations, specified by maps $G \rightarrow \mathbf{L}^\C( \C, \C ) \cong \widehat{\mathbb{T}}$. Thus the classifying map of the original representation factors through $\widehat{\mathbb{T}} \times_{C_2} \widehat{\mathbb{T}} \times_{C_2} ... \times_{C_2} \widehat{\mathbb{T}} \cong \widehat{\mathbb{T}^n}$. This factorization gives the desired embedding.

(2) $\implies$ (3) Note that (3) holds for $\widehat{\mathbb{T}^n}$. Furthermore, if (3) holds for an augmented group, it holds for any augmented subgroup. Thus (3) holds for any augmented subgroup of $\widehat{\mathbb{T}^n}$.

(3) $\implies$ (4) If $G$ does not have the trivial augmentation, then we can choose an order two element of $G$ which augments to the nontrivial element of $C_2$. This specifies a semidirect product expression of $G$ as $G \rtimes C_2$.

(4) $\implies$ (1) In either case, any Real $\eta$-representation is determined by a map $G \rightarrow \mathbf{L}^\C(\C^n,\C^n)$. The map of underlying non-augmented maps is $Ker(\eta) \rightarrow U(n)$, and since $Ker(\eta)$ is abelian, the image is contained in a maximal torus of $U(n)$. Thus $G \rightarrow \mathbf{L}^\C(\C^n,\C^n)$ factors through $\widehat{\mathbb{T}^n}$. This factorization induces a direct sum decomposition into one-dimensional representations. In particular, any irreducible representation is one-dimensional.
\end{proof}

Next, we will relate the character groups and representation rings of abelian compact Lie groups $G$ with their associated augmented compact Lie groups $\hat{G} = G \rtimes C_2$. 

\begin{definition}
    A Real character of an augmented compact Lie group $\hat{G} = G \rtimes C_2$ is an augmented map $\hat{\alpha}: \hat{G} \to \widehat{\mathbb{T}}$. We write $\widehat{\text{Hom}}(\hat{G},\widehat{\mathbb{T}})$ for the Real character group of $\hat{G}$.
\end{definition}

Suppose $\hat{G} = G \rtimes C_2$, and suppose $\hat{\alpha}:\hat{G} \to \widehat{\mathbb{T}}$ is a Real character of $\hat{G}$. Then the restriction $\hat{\alpha}\mid_G$ factors through $\mathbb{T} \subset \widehat{\mathbb{T}}$. In other words, $\hat{\alpha}$ determines a complex character $\alpha:G \to \mathbb{T}$. This determines a group homomorphism
\[ \begin{tikzcd} 
\widehat{\text{Hom}}(\hat{G},\widehat{\mathbb{T}}) \ar[r,"\psi"] &  \text{Hom}(G,\mathbb{T})
\end{tikzcd} \]
from the Real character group of $\hat{G}$ to the complex character group of $G$. The kernel of $\psi$ consists of those Real characters $\hat{\alpha} : \hat{G} \to \widehat{\mathbb{T}}$ which map $G \subset \hat{G}$ to $1 \in \mathbb{T}$. Such a map is determined by the image of $(e,\gamma)$, which may be any element of the form $(w,\gamma)$ for $w \in \mathbb{T}$. In other words, the kernel of $\psi$ is isomorphic to $\mathbb{T}$. Moreover, the map $\psi$ admits a section: If $\alpha: G \to \mathbb{T}$ is a complex character of $G$, then we may apply $- \rtimes C_2$ to obtain a Real character $\hat{\alpha} = \alpha \rtimes C_2$ of $\hat{G}$, whose restriction to $G \subset \hat{G}$ is $\alpha$. We record our observations below.

\begin{lemma}\label{lemma:short-exact-sequence}
    If $G$ is an abelian compact Lie group, and $\hat{G} = G \rtimes C_2$, then there is a canonically split short exact sequence 
    \[ \begin{tikzcd} 
    1 \ar[r] & \mathbb{T} \ar[r] & \widehat{\text{Hom}}(\hat{G},\widehat{\mathbb{T}}) \ar[r,"\psi"] & \text{Hom}(G,\mathbb{T}) \ar[r] & 1.
    \end{tikzcd} \]
\end{lemma}

While there are strictly more Real characters of $\hat{G}$ than there are complex characters of $G \subset \hat{G}$, the following proposition implies that up to isomorphism, irreducible Real representations of $\hat{G}$ do correspond bijectively to irreducible complex representations of $G$ via restriction.

\begin{lemma}\label{lemma:isomorphic-real-reps}
    Suppose $\hat{G} = G \rtimes C_2$, and suppose  $\hat{\alpha},\hat{\beta}:\hat{G} \to \widehat{\mathbb{T}}$ are Real characters of $\hat{G}$. If $\hat{\alpha} \mid_G = \hat{\beta} \mid_G$, then $\hat{\alpha}$ and $\hat{\beta}$ are isomorphic as Real representations of $\hat{G}$.
\end{lemma}
\begin{proof}
    We know that $\hat{\alpha}(e,\gamma) = (\zeta_1,\gamma)$ and $\hat{\beta}(e,\gamma) = (\zeta_2,\gamma)$ for some $\zeta_1,\zeta_2 \in \mathbb{T}$. We claim that the map $\varphi:\mathbb{C} \to \mathbb{C}$ given by $\varphi(z) = \sqrt{\zeta_2\zeta_1^{-1}}$ specifies an isomorphism $\hat{\alpha} \cong \hat{\beta}$. We know that $\hat{\alpha}$ and $\hat{\beta}$ agree on $G \subset \hat{G}$, so it suffices to check
    \begin{align*}
    \varphi(\hat{\alpha}(e,\gamma) z) = \varphi((\zeta_1,\gamma) z) = \varphi(\zeta_1\bar{z}) & = \sqrt{\zeta_2\zeta_1^{-1}}\zeta_1\bar{z}\\
    & = \zeta_2 \overline{\sqrt{\zeta_2\zeta_1^{-1}} z}\\
    & = \hat{\beta}(e,\gamma) \varphi(z).
    \end{align*}
\end{proof}

For a compact Lie group $G$, we write $RO(G)$ for the real representation ring of $G$, and $RU(G)$ for the complex representation ring of $G$. For an augmented compact Lie group $\hat{G}$, we will write $RR(\hat{G})$ for the Real representation ring of $\hat{G}$, which is the Grothendieck ring of isomorphism classes of Real $\hat{G}$-representations. If $\hat{G}$ is an augmented compact Lie group with $G = \text{Ker}(\hat{G} \to C_2)$, then we may define a map $RR(\hat{G}) \to RU(G)$ by restricting the action of $\hat{G}$ on $V$ to the subgroup $G \subset \hat{G}$. We call this the {\it restriction} map.

The category of $\eta$-representations has a symmetric monoidal structure, essentially given by tensor product. To define this, it suffices by universality to define an appropriate $\mathbf{L}^\C(\C^{nk},\C^{nk})$ action on $\C^n \otimes \C^k$. To construct this, observe the isomorphism $\C^n \otimes \C^k \cong \C^{nk}$ obtained by the choice of basis $e_i \otimes e_j$ for $\C^n \otimes \C^k$ (where $i = 1,...,n$ and $j = 1,...,k$). Now pull back the action of $\mathbf{L}^\C(\C^{nk},\C^{nk})$ on $\C^{nk}$ via this isomorphism.

\begin{definition}
    Let $\eta : G \rightarrow C_2$ any augmented Lie group. Define the representation ring $RR(\eta)$ as the Grothendieck group completion of the additive monoid of finite dimensional $\eta$-representations under direct sum. It becomes a ring under tensor product.
\end{definition}

\begin{lemma}[\cite{Sch14}]
\label{lem:RR-RU-identification}
     Let $G = A \rtimes C_2$ with $A$ abelian. Then the restriction (forgetful) map $RR(\eta)\to RU(A)$ is an isomorphism.
\end{lemma}

\begin{proof}
    By \cite{Sch14}, the composite \[ RR(\eta) \xrightarrow{res} RU(A) \xrightarrow{tr} RR(\eta) \]
    is multiplication by $2$. Since each group is free abelian, we deduce that restriction is injective.

    To show that restriction is surjective, take any character $\alpha: A \to U(1)$. Applying $-\rtimes C_2$ we can extend it to a map \[ \widehat{\alpha} : G \to \widehat{U(1)}. \] We have $res^{G}_A (\widehat{\alpha})=\alpha$, which concludes the proof.
\end{proof}

\begin{definition}
    If $\eta : G \rightarrow C_2$ is an augmented compact Lie group, then a complete Real $\eta$-universe is a Real representation $\mathcal{U}_{\eta}$ of $\eta$ which is of countable dimension, and into which every finite-dimensional Real representation of $\eta$ embeds.
\end{definition}

\begin{lemma}[\cite{Sch14}]
\label{lem:existence-of-complete-Real-universes}
    If $\eta$ is quasi-abelian, then there exists a complete Real $\eta$-universe $\mathcal{U}_{\eta}$. 
\end{lemma}

\begin{proof}
If $\eta : G \rightarrow C_2$ is the trivial augmentation, then $G$ is abelian, and a Real $\eta$-representation is the same thing as a real (ie over $\mathbb{R}$) $G$-representation. Thus a complete Real universe of $\eta$ is simply a complete real universe of $G$, which is known to exist.

If $\eta$ is nontrivial, then $G = A \rtimes C_2$ with $A$ abelian. Start with a complete complex universe $\mathcal{U}_A$ for $A$, and write it as a sequential colimit over finite-dimensional sub-$A$-representations $V_i$. Such a representation is classified by a map $A \rightarrow U(\C^n)$ for some $n$, and we may apply $- \rtimes C_2$ to obtain a map $G \rightarrow U(\C^n) \rtimes C_2 = \mathbf{L}^\C(\C^n,\C^n)$ describing a Real $\eta$ representation $\widehat{V_i}$. Using Lemma \ref{lem:RR-RU-identification} one may then show that the sequential colimit of the spaces $V_i$ will then be a complete Real $\eta$-universe.
\end{proof}

\subsection{Real global homotopy theory}

We are interested in the special case of Schwede's $C$-global homotopy theory developed in some detail in \cite{Sch22} in the case $C = C_2$. To start with, define an orthogonal $C_2$-spectrum as an orthogonal spectrum equipped with an action of $C_2$ by automorphisms of orthogonal spectra. Given such a spectrum $X$ and an augmented group $\eta : G \rightarrow C_2$, it has $G$-equivariant homotopy groups $\pi_k^G(\eta^* X)$ defined as in \cite[3.1.11]{Sch18}.

\begin{definition}
    A $C_2$-global equivalence is a map $f : X \rightarrow Y$ of orthogonal $C_2$-spectra such that for every augmented compact Lie group $\eta : G \rightarrow C_2$ we have \[ \pi_k^G( \eta^* f) : \pi_k^G(\eta^* X) \rightarrow \pi_k^G(\eta^* Y) \] is an isomorphism.
\end{definition}

We aim to invert such equivalences to form a homotopy category.

\begin{theorem}[{\cite{Sch22}}]
\label{thm:cofibration-cat-structure}
    There is a cocomplete stable cofibration category structure on the category of orthogonal $C_2$-spectra such that the weak equivalences are the $C_2$-global equivalences.
\end{theorem}

\begin{definition}
    The Real global homotopy category is defined as the homotopy category associated to the cofibration category structure of Theorem \ref{thm:cofibration-cat-structure}. We denote it by $\mathcal{GH}_{C_2}$ in accordance with \cite{Sch22}.
\end{definition}

Now we may unpack some consequences of the definitions. First, a Real global spectrum $E$ is just an orthogonal $C_2$-spectrum, viewed as an object of the cofibration category of Theorem \ref{thm:cofibration-cat-structure}. For each quasi-abelian augmented Lie group $\eta : G \rightarrow C_2$, the pullback functor $\eta^*:Sp^{C_2}\to Sp^G$ already sends $C_2$-global equivalences to genuine $G$-equivariant homotopy equivalences, thus naturally descends to a functor on the homotopy categories $\eta^*:\mathcal{GH}_{C_2}\to G-\mathcal{GH}$. 

In particular, if $E$ is a Real global spectrum, then we get, for each augmented quasi-abelian $\eta : G \rightarrow C_2$, a genuine $G$-spectrum $\eta^{*}E=E_\eta$. When the augmentation is obvious we may sometimes write $E_G$ instead of $E_{\eta}$, for example when we have the trivial augmentation $G \rightarrow C_2$, or we may write $E_{\hat{G}}$ for the projection augmentation $\hat{G} := G \rtimes C_2 \rightarrow C_2$.

The point-set level functor $\eta^*:Sp^{C_2}\to Sp^G$ is strong symmetric monoidal, since the smash products in both categories are defined by smash products on the underlying spectra equipped with diagonal group actions. This will be enough for us to deduce each $M\R_{\eta}$ is a commutative ring in orthogonal $G$-spectra. We believe the derived functor of $\eta^*$ is strong symmetric monoidal, but we do not need this fact. (One ought to be able to prove this via some results analogous to that of \cite[p448]{Sch18}, by showing a $C_2$-flat ring, given the pullback $G$-action, is $G$-flat.)

If $E$ is a Real global spectrum, then each cohomology theory $E_\eta^*$ has a natural extension to an $RO(C_2)$-grading. Namely, $E_\eta^*$ has a natural $RO(G)$-grading, and we have $\eta^* : RO(C_2) \rightarrow RO(G)$. If $\eta$ is the trivial augmentation, then letting $|\alpha|$ denote the virtual dimension of $\alpha \in RO(C_2)$, we have $E_\eta^{\alpha}(-) \cong E_\eta^{|\alpha|}(-)$.

\begin{proposition}
\label{prop:RO(C_2)-grading-is-well-defined}
    The $RO(C_2)$ grading on $E_\eta^*$ is compatible with restriction along morphisms of augmented groups.
\end{proposition}

\begin{proof}
    If $f : H \rightarrow G$ is a morphism of augmented groups, then the formula \[ f^* (E_\eta^\alpha(-)) \cong (E_{\eta \circ f})^{f^* \alpha}(-) \] for $\alpha \in RO(G)$ implies that this grading is preserved by restriction along any morphism of augmented groups. 
\end{proof}

\section{Geometric constructions with augmented Lie groups}

We will require geometric constructions analogous to those in \cite[Section 3.2]{Hau22}, but in the setting of augmented Lie groups. We outline their construction here.

\subsection{The induction isomorphism}

We will start with induction isomorphisms. Following \cite[Section 3.2]{Hau22}, we can restrict and induct along homomorphisms of groups. Given a homomorphism of augmented groups, we can therefore restrict and induct along the underlying group homomorphisms (by forgetting the data of the augmentations). In particular, if we have a Real global spectrum $E$, then we obtain induction maps $ind_d$ analogous to those utilized throughout \cite{Hau22}. 

Recall from \cite[Section 3.2]{Hau22} that if $f : G \rightarrow H$ is a group homomorphism and $X$ is a based $G$-space, then induction along $f$ given by 
\begin{equation}\label{induction along f}
    \text{ind}_f(X) = H_+ \wedge_f X.
\end{equation}
It defines a functor $Top^G_*\to Top^H_*$ which is left adjoint to restriction along $f$, given by 
\begin{equation}\label{restriction along f}
    \text{res}_f(X) = f^*X.
\end{equation}
If $E$ is a global spectrum and $X$ is a $G$-space, then the induction homomorphism is the composite 

\[\begin{tikzcd} ind_f: E^*_H(H_+ \wedge_f X) \ar[r,"f^*"] &  E_G^*( f^*(H_+ \wedge_f X))  \ar[rr, "(\beta_X)^*"] & &  E_G^*(X),
\end{tikzcd} \]
where $\beta_X: X \to f^*(H_+ \wedge_f X)$ is the unit of the adjunction. 

Now if $\eta : G \rightarrow C_2$ is an augmented group, by a $\eta$-CW complex we just mean a $G$-CW complex in the usual sense, where $G$ is the underlying Lie group. We define $\eta$-spaces similarly. Let $\eta : G \rightarrow C_2$ and $\epsilon : H \rightarrow C_2$. If $f : G \rightarrow H$ is an augmented group homomorphism, we define induction and restriction along $f$ as in \eqref{induction along f} and \eqref{restriction along f} by forgetting the augmentation of the groups. Now let $E$ be a Real global spectrum and $X$ a based $\eta$-space, we define the induction homomorphism

\[\begin{tikzcd} ind_f: E^{\bigstar}_{\epsilon}(H_+ \wedge_f X) \ar[r,"f^*"] &  E_{\eta}^{\bigstar}( f^*(H_+ \wedge_f X))  \ar[rr, "(\beta_X)^*"] & &  E_{\eta}^{\bigstar}(X).
\end{tikzcd} \]

Since $f^*$ preserves the $RO(C_2)$-grading by Proposition \ref{prop:RO(C_2)-grading-is-well-defined} and $(\beta_X)^*$ preserves it as $\beta_X$ is a map of $G$-spaces, we see that $ind_f$ preserves the $RO(C_2)$-grading as well.

Recall that the kernel of an augmented map $\phi: G \to H$ is simply the kernel $Ker(\phi)$ of the group homomorphism $G \to H$, equipped with canonical augmentation $Ker(\phi) \subset G \to C_2$. The following lemma is adapted from \cite[Proposition 3.3.8]{DHL+23} to our augmented setting.

\begin{lemma}\label{lem induction iso}
     Let $E$ be a Real global spectrum. Let $\alpha : K \rightarrow G$ a surjective map of augmented compact Lie groups. Let $X$ be a based $K$-CW complex on which $Ker(\alpha)$ acts freely in the based sense. Then the induction map\[ \begin{tikzcd} ind_{\alpha} : E^{\bigstar}_{\eta}(G_+ \wedge_{\alpha} X) \ar[r,"\cong"] & E^{\bigstar}_{\eta \circ \alpha}(X)\end{tikzcd}  \]
    is an isomorphism for ${\bigstar}\in RO(C_2)$. 
 \end{lemma}

\begin{proof}
    We follow the proof of \cite[Proposition 3.3.8]{DHL+23}. The functor $G_+ \wedge_{\alpha}-$ preserves equivariant homotopy and commute with wedges and mapping cones. So the functor $E^*_{\eta}(G_+\wedge_{\alpha}-)$ from the category of $K$-spaces to graded abelian groups defines a cohomology theory. The induction map
    \[
    ind_{\alpha}: E^*_{\eta}(G_+ \wedge_{\alpha} X) \to E_{\eta \circ \alpha}^*(X)
    \]
    for a natural transformation of cohomology theories, and we only need to prove the claim for $X= K/L$, where $L\subset \hat{K}$ is compact subgroup with $L\cap Ker(\alpha)= \{ e \}$. The restriction
    \[
    \bar{\alpha}=\alpha|_{L}:L\to G
    \]
    is injective. The $G$-map
    \[
    \psi: G_+\wedge_{\bar{\alpha}} L/L_+\to G_+\wedge_{\alpha} K/L_+,\,[g,eL]\mapsto [g,eL]
    \]
    is a homeomorphism. The two commutative diagrams in \cite[Proposition 3.3.8]{DHL+23} also hold in our context after adjusting everything to the based context. Thus $ind_{\alpha}$ is an isomorphism on $\mathbb{Z}$-graded cohomology theories. If $K$ and thus $G$ are trivially augmented, then the $RO(C_2)$-grading degenerates to $\mathbb{Z}$-grading and we have finished the proof.

    Now suppose $K$ and thus $G$ are non-trivially augmented. By construction, the induction map preserves $RO(C_2)$-grading and we have the following commutative diagram with exact rows. Here we use $E_{Ker(\eta)}$ to denote the restriction of $E_\eta$ to $E_{Ker(\eta)}$, and likewise for $E_{Ker(\eta \circ \alpha)}$.
    \[\scriptsize\begin{tikzcd}
	\cdots & {E^{*+n\sigma}_{Ker(\eta)}(i^*_{Ker(\eta)}(G_+\wedge_{\alpha}X))} && {E^{*+n\sigma}_{G}(G_+\wedge_{\alpha}X)} && {E^{*+(n+1)\sigma}_{G}(G_+\wedge_{\alpha}X)} & \cdots \\
	\\
	\cdots & {E^{*+n\sigma}_{Ker(\eta \circ \alpha)}(i^*_{Ker(\eta \circ \alpha)}X)} && {E^{*+n\sigma}_{K}(X)} && {E^{*+(n+1)\sigma}_{K}(X)} & \cdots
	\arrow[from=1-1, to=1-2]
	\arrow["{tr^{G}_{Ker(\eta)}}", from=1-2, to=1-4]
	\arrow["{ind_{\alpha|_{Ker(\eta \circ \alpha)}}}", from=1-2, to=3-2]
	\arrow["{a_{\sigma}}", from=1-4, to=1-6]
	\arrow["{ind_{\alpha}}", from=1-4, to=3-4]
	\arrow[from=1-6, to=1-7]
	\arrow["{ind_{\alpha}}", from=1-6, to=3-6]
	\arrow[from=3-1, to=3-2]
	\arrow["{tr^{K}_{Ker(\eta \circ \alpha)}}", from=3-2, to=3-4]
	\arrow["{a_{\sigma}}", from=3-4, to=3-6]
	\arrow[from=3-6, to=3-7]
\end{tikzcd}\]
Note that the restriction $\alpha|_{Ker(\eta \circ \alpha)}$ is again surjective. Also, $a_{\sigma}$ denotes the Euler class of $\sigma$. The left most map is $ind_{\alpha|_{Ker(\eta \circ \alpha)}}$, since we have the $Ker(\eta)$-homeomorphism
\[
Ker(\eta)_+\wedge_{\alpha|_{Ker(\eta \circ \alpha)}}i^*_{Ker(\eta \circ \alpha)} X \xrightarrow{\cong}i^*_{Ker(\eta)}(G_+\wedge_{\alpha}X), [g,x]\mapsto [g,x].
\]
The left most square commutes since transfers commute with restrictions along surjective maps (this is true for any genuine $G$-spectrum). By induction on $n\geq 0$, we can prove that $ind_{\alpha}$ is an isomorphism for all $n \geq 0$.

Using the following commutative diagram with exact rows
\[\scriptsize\begin{tikzcd}
	\cdots & {E^{*+n\sigma}_{Ker(\eta)}(i^*_{Ker(\eta)}(G_+\wedge_{\alpha}X))} && {E^{*+n\sigma}_{G}(G_+\wedge_{\alpha}X)} && {E^{*+(n-1)\sigma}_{G}(G_+\wedge_{\alpha}X)} & \cdots \\
	\\
	\cdots & {E^{*+n\sigma}_{Ker(\eta \circ \alpha)}(i^*_{Ker(\eta \circ \alpha)}X)} && {E^{*+n\sigma}_{K}(X)} && {E^{*+(n-1)\sigma}_{K}(X)} & \cdots
	\arrow[from=1-2, to=1-1]
	\arrow["{ind_{\alpha|_{Ker(\eta \circ \alpha)}}}", from=1-2, to=3-2]
	\arrow["{res^G_{Ker(\eta)}}"', from=1-4, to=1-2]
	\arrow["{ind_{\alpha}}", from=1-4, to=3-4]
	\arrow["{a_{\sigma}}"', from=1-6, to=1-4]
	\arrow["{ind_{\alpha}}", from=1-6, to=3-6]
	\arrow[from=1-7, to=1-6]
	\arrow[from=3-2, to=3-1]
	\arrow["{res^{K}_{Ker(\eta \circ \alpha)}}"', from=3-4, to=3-2]
	\arrow["{a_{\sigma}}"', from=3-6, to=3-4]
	\arrow[from=3-7, to=3-6]
\end{tikzcd}\]
and inducting on $n\leq 0$ proves the conclusion for $n\leq 0$.
\end{proof}

As a corollary, we have the following

\begin{corollary}
\label{corollary HauRemark3.3}
    Let $E$ be a Real global spectrum. Let $\hat{K} = K \rtimes C_2$, $G$ any augmented compact Lie group, $X$ a based $G$-CW space and $p : G \times_{C_2} \hat{K} \rightarrow G$ the projection. Then $K = ker(p)$ acts freely on $(G \times_{C_2} \hat{K})_+ \wedge_{G} X$, and consequently there is an induction isomorphism \[ E^{\bigstar}_{G \times_{C_2} \hat{K}}((G \times_{C_2} \hat{K})_+ \wedge_{G} X) \cong E^{{\bigstar}}_{\eta}(((G \times_{C_2} \hat{K})_+ \wedge_{G} X)/K) \cong E_{\eta}^{\bigstar}(X) \] determined by $p$. In the space $(G\times_{C_2}\hat{K})_+\wedge_{G}X$, we use the inclusion $G\cong G\times_{C_2}C_2\hookrightarrow G\times_{C_2}\hat{K}$ to view $G$ as a subgroup of the later. This isomorphism agrees with the induction isomorphism determined by $G\cong G\times_{C_2}C_2 \hookrightarrow G \times_{C_2} \hat{K}$ using the canonical inclusion $C_2\to \hat{K}$. Here ${\bigstar}\in RO(C_2)$.
\end{corollary}
\begin{proof}
The first isomorphism comes from the induction isomorphism along the projection $p : G \times_{C_2} \hat{K} \rightarrow G$ from Lemma \ref{lem induction iso}. Note that the conditions of this lemma is satisfied: By our assumption on $\hat{K}$, we have that the projection $p : G \times_{C_2} \hat{K} \to G$ is split surjective, where we get the splitting by applying $G\times_{C_2}-$ to the split inclusion $C_2\to \hat{K}$. 

The second isomorphism comes from the $G$-homeomorphisms 
\[
G\wedge_p((G \times_{C_2} \hat{K})_+ \wedge_{G} X))\cong ((G \times_{C_2} \hat{K})_+ \wedge_{G} X))/K\cong X.
\]   
\end{proof}

Since we will only require the case that $\hat{K}$ is an augmented torus $\widehat{\mathbb{T}}$, the condition on $\hat{K}$ is not overly restrictive. 

\subsection{A model for homotopy orbits of augmented groups}\label{Section homotopy orbit construction}

In this section we define the notion of a universal $K$-space in $G$-spaces, for augmented groups $K$ and $G$. Then we define the notion of $(K,G)$-homotopy orbits using the universal spaces. Note that these definitions are different than the ones given in \cite{Hau22}. This is because we want the unit sphere $S(\mathcal{U}_{\eta})$ in a complete Real $\eta$-universe $\mathcal{U}_{\eta}$ to be a model for the universal $\widehat{\mathbb{T}}$-space in $\eta$-spaces $E_{\eta} \widehat{\mathbb{T}}$.

\begin{lemma}
\label{GraphSubgroup}
    Let $\eta : G \rightarrow C_2$ and $\epsilon : K \rightarrow C_2$ be two augmented groups. A group homomorphism $\alpha : G \rightarrow K$ respects the augmentation if and only if the associated graph subgroup $H \subset G \times K$ is contained in the fiber product $G \times_{C_2} K$.
\end{lemma}

\begin{proof}
    Given $\alpha$, the associated graph subgroup $H$ may be described as the set of pairs $(g,\alpha(g)) \in G \times K$. Respecting the augmentation is equivalent to the statement that $g$ and $\alpha(g)$ map to the same element of $C_2$. The fiber product is described set-theoretically as the collection of elements $(g,k)$ which map to the same element of $C_2$, so the result follows.
\end{proof}

\begin{definition}
Let $\eta : G \rightarrow C_2$ and $\epsilon : K \rightarrow C_2$ be augmented groups. Then a universal $K$-space in $G$-spaces is a $G \times_{C_2} K$ space $E_{G} K$ such that the $H$-fixed points $(E_{G} K)^{H}$ are contractible for $H \subset G \times_{C_2} K$ the graph subgroup associated to an augmented homomorphism $\alpha : L \rightarrow K$ from an subgroup $L$ (given the canonical augmentation $L \subset G \to C_2$) of $G$, and empty otherwise. The condition on $H$ is equivalent to $H \cap (\{1\}\times Ker(\epsilon))=\{ 1\}\times \{1\}$.   
\end{definition}

Given such a space $E_G K$, we can form a functor, called $(K,G)$-homotopy orbits. This functor takes in a $G \times_{C_2} K$-space $X$ and produces the $G$-space $X_{h_G K} := E_G K_+ \wedge_{Ker(\epsilon)} X$. Here $Ker(\epsilon)$ is viewed as a subgroup of $G \times_{C_2} K$ by the inclusion as the kernel of the projection to $G$. Note that there is a slight conflict of notation with that of \cite[Section 3.3]{Hau22}. Specifically, if $G$ and $K$ are regarded just as compact Lie groups, then Hausmann constructs $(K, G)$-homotopy orbits, which are distinct from ours unless $G$ and $K$ are given the trivial augmentations. We will only use our construction of homotopy orbits, so there should be little chance for confusion.

In the following, we provide an explicit model for the universal $\widehat{\mathbb{T}}$-space in $G$-spaces, $E_G \widehat{\mathbb{T}}$.

\begin{lemma}
    Let $\eta : G \rightarrow C_2$ be a quasi-abelian group. Then the unit sphere $S(\mathcal{U}_\eta)$ in a complete Real $\eta$-universe $\mathcal{U}_\eta$ is a model for the universal $\widehat{\mathbb{T}}$-space in $G$-spaces $E_G \widehat{\mathbb{T}}$, where $\widehat{\mathbb{T}}$ acts on $\mathcal{U}_\eta$ by combining the action of $\mathbb{T}$ by scalar multiplication with the action of $C_2$ describing the Real structure.
\end{lemma}

\begin{proof}
    We must check that for an augmented subgroup $H \subset G \times_{C_2} \widehat{\mathbb{T}}$, the $H$-fixed points $S(\mathcal{U}_\eta)^H$ are contractible if $H$ is the graph subgroup associated to a homomorphism from an augmented subgroup of $G$, and empty otherwise. Note that this condition on $H$ is equivalent to the condition that $H$ is an augmented subgroup of $G \times_{C_2} \widehat{\mathbb{T}}$ such that $H \cap (\{ 1 \} \times {\mathbb{T}}) = \{ 1 \}\times \{ 1 \}$ by Lemma \ref{GraphSubgroup}.

    Choose a complete flag for $\mathcal{U}_\eta$. This specifies a basis over $\R$ for which $Ker(\eta)$ and $\mathbb{T}$ act by $2 \times 2$ block diagonal matrices, and $C_2$ acts diagonally (heuristically, our basis over $\R$ comes from the imaginary and real parts of a basis over $\C$). If $H \cap (\{ 1 \} \times {\mathbb{T}}) \neq \{ 1 \}\times \{ 1 \}$, then it immediately follows that there are no $H$-fixed points.

    Conversely, assume $H$ is the graph subgroup associated to some $\alpha : K \rightarrow \widehat{\mathbb{T}}$. Let $\beta_i$ run through the nonempty set of all irreducible Real $\eta$-representations $G \rightarrow \hat{\mathbb{T}}$ such that $\alpha(k) \cdot \beta_i(k) = 1$, and let $\mathcal{V}$ be the subrepresentation of $\mathcal{U}_\eta$ spanned by all irreducible representations $\beta_i$. Note that $C_2 \subset \mathbb{T}$ acts on $\mathcal{V}$ by complex conjugation. If $H$ has the trivial augmentation, then the $H$-fixed points form the unit sphere in $\mathcal{V}$, which is contractible. Otherwise the $H$-fixed points form the unit sphere in $\mathcal{V}^{C_2}$, which is again contractible.
\end{proof}

For a Real global spectrum $E$, augmented groups $\eta : G \rightarrow C_2$ and $\hat{K} = K\rtimes C_2$, and a $ G \times_{C_2} \hat{K}$-space $X$,
we can form a map \[ h_{G,\hat{K}}(X) : E^{\bigstar}_{G \times_{C_2} \hat{K}}(X) \rightarrow E^{\bigstar}_\eta ( X_{h_G \hat{K}} ) \] by composing \[ (E_G \hat{K} \rightarrow *)^{\bigstar} : E^{\bigstar}_{G \times_{C_2} \hat{K}}(X) \rightarrow E^{\bigstar}_{G \times_{C_2} \hat{K}} ((E_G \hat{K})_+ \wedge X) \] with the inverse of the induction isomorphism \[ E^{\bigstar}_{G \times_{C_2} \hat{K}} ((E_G \hat{K})_+ \wedge X) \rightarrow E^{\bigstar}_\eta(( E_G \hat{K})_+ \wedge_K X) \]
Now suppose we have a map of augmented groups $\alpha : G \rightarrow \hat{K}$ with corresponding graph subgroup $H$. If we choose a point $*_{\alpha}$ in the contractible space $(E_G \hat{K})^H$, then $x \mapsto *_\alpha \wedge x$ specifies a $H$-map $(Id_G,\alpha)^*(X) \rightarrow X_{h_G \hat{K}}$ sending $x$ to the class $[*_\alpha,x]$.

The following theorem is adapted from \cite[Lemma 3.4]{Hau22} to the augmented setting.

\begin{lemma}
\label{HauLemma3.4}
    Assume $\hat{K} = K \rtimes C_2$. The composition \[ E^{\bigstar}_{G \times_{C_2} \hat{K}}(X) \xrightarrow{h_{G,\hat{K}}(X)} E^{\bigstar}_\eta (X_{h_G \hat{K}}) \rightarrow E^{\bigstar}_\eta ((Id_G,\alpha)^*(X)) \] equals the restriction map $(Id_G,\alpha)$. Here $\bigstar=RO(C_2)$.
\end{lemma}

\begin{proof}
    Following Hausmann, we consider the graph subgroup $i : H \hookrightarrow G \times_{C_2} \hat{K}$ associated to $\alpha : G \rightarrow \hat{K}$. This inclusion, followed by the projection $p_1: G \times_{C_2} \hat{K} \to G$ specifies an isomorphism $H \cong G$ of augmented groups. The map $*_\alpha \wedge - : X \rightarrow E_G \hat{K}_+ \wedge X$ is equivariant with respect to the action of $H$, induces up to a $G \times_{C_2} \hat{K}$-equivariant map \[ (G \times_{C_2} \hat{K})_+ \wedge_H X \rightarrow E_G \hat{K}_+ \wedge X .\]

    Now $K$ acts freely on the domain of the map above, and the corresponding quotient $((G \times_{C_2} \hat{K})_+ \wedge_H X)/K$ is just $(Id_G,\alpha)^*(X)$ (one can see this by noticing $i(H)$ and $K$ generates $G \times_{C_2} \hat{K}$). Now we can form the diagram 
    \[ \begin{tikzcd}
    E^{\bigstar}_{G \times_{C_2} \hat{K}}(X) \arrow{r}{(E_G \hat{K}\to *)^{\bigstar}} \arrow{dr} & E^{\bigstar}_{G \times_{C_2} \hat{K}}(E_G \hat{K}_+ \wedge X) \arrow{d}{(*_\alpha \wedge -)^*} & E^{\bigstar}_\eta (E_G \hat{K}_+ \wedge_K X) \arrow{l}{ind_\alpha} \arrow{d}{(*_\alpha \wedge -)^*} \\
    {} & E^{\bigstar}_{G \times_{C_2} \hat{K}}((G \times_{C_2} \hat{K})_+ \wedge_{H} X) & E^{\bigstar}_\eta((Id_G,\alpha)^*(X)) \arrow{l}{ind_\alpha}
    \end{tikzcd} \]
    where the left diagonal map is induced by the counit of the forgetful-induction adjunction. Going from the top left to the bottom right is the composite which we wish to show is restriction along $(Id_G,\alpha)$, so it suffices to check that the bottom composition is this restriction. This follows directly from Corollary \ref{corollary HauRemark3.3} and \cite[Lemma 3.2]{Hau22}.
\end{proof}

\section{Real orientations of global spectra and equivariant spectra}\label{Section 4 Real orientations of global spectra and equivariant spectra}

In this section, we generalize definitions from \cite{CGK00} and \cite{Hau22} and define Real orientations of equivariant spectra and Real global spectra. We show that a Real orientation of a Real global spectrum $E$ induces a Real $\eta$-orientation of the $G$-spectra $E_\eta$ for all compact quasi-abelian Lie groups $\eta : G \rightarrow C_2$. 

We then define the concept of a Real global group law and show that the $RO(C_2)$-graded homotopy groups $\pi_{\bigstar}^{\widehat{\mathbb{T}^n}}(E_{\widehat{\mathbb{T}^n}})$ assemble to a Real global group law. We will eventually apply all of the results of this section to the Real global spectrum $\mathbf{MR}$, constructed in the next section.

\subsection{Real $\eta$-orientations and Real global orientations}

Before defining our Real $\eta$-orientations, we need to define a Real analogue of complex stability as in \cite{CGK02}, which we call {\it Real stability}. 

\begin{definition}
    Let $\eta : G \rightarrow C_2$ be any augmented compact Lie group. If $E_\eta$ is a homotopy commutative $G$-ring spectrum, then a Real stable structure on $E_\eta$ is a choice of isomorphism $E^{V}_\eta(X) \cong E^{(dim_\C V) \rho}_\eta(X)$ for each Real representation $V$ of $\eta$ and is natural in the pointed $G$-spaces $X$. We require these isomorphisms to be transitive and given by a multiplication by a unit in $E_\eta^{V-(dim_{\C}V)\rho}(S^0)$. If we have chosen a Real stable structure on $E_\eta$, we say that $E_\eta$ is Real stable.
\end{definition}

\begin{remark}
    A Real stable structure on a homotopy $G$-ring spectrum $E_\eta$ determines a complex stable structure on $E_{Ker(\eta)} := Res^G_{Ker(\eta)} E_\eta$ as follows. Restriction is a ring homomorphism, a unit in $E_\eta^{V-(dim_{\C}V)\rho}(S^0)$ restricts to a unit in $E_{Ker(\eta)}^{V-(dim_{\C}V)}(S^0)$.
\end{remark}

For an augmented compact Lie group $\eta : G \rightarrow C_2$, let $\rho$ be the Real representation of $\eta$ obtained by pulling back the complex conjugation action of $C_2$ on $\mathbb{C}$ along the augmentation $G \to C_2$. If $\alpha$ is an irreducible Real representation of $\eta$, then there is an isomorphism of $G$-spaces $\C P(\rho \oplus \alpha) \cong S^{\alpha}$ where $[1:v]$ corresponds to $v\in\alpha$ and $[0:v]$ correspond to $\infty$.

\begin{definition}\label{Definition Real orientation}
    If $E_\eta$ is a homotopy commutative $G$-ring spectrum, then a Real $\eta$-orientation of $E_\eta$ is a class $y(\rho) \in E_\eta^\rho(\C P(\mathcal{U}_\eta))$ which restricts to a $RO(G)$-graded unit of \[ E_\eta^\rho( \C P(\rho \oplus \alpha) ) \cong E_\eta^\rho(S^\alpha)\]
    where $\alpha$ ranges through the non-trivial irreducible Real $\eta$-representations, and restricts to the multiplicative identity when $\alpha = \rho$.
\end{definition}

\begin{remark}
\label{rem:Real-ori-determines-Real-stable-structure}
    A Real $\eta$-orientation of $E_\eta$ determines a Real stable structure on $E_\eta$.
\end{remark}

\begin{proposition}
    Let $\hat{G} = G \rtimes C_2$ for $G$ compact abelian Lie, and $E_\eta$ a homotopy commutative ring $\hat{G}$-spectrum. Then a Real $\eta$-orientation $y(\rho)\in E_\eta^\rho(\C P(\mathcal{U}_\eta))$ restricts to a complex orientation of $y(\epsilon)\in E_{G}^2(\C P(\mathcal{U}_G))$. Here $\epsilon$ is the one-dimensional trivial complex $G$-representation.
\end{proposition}

\begin{proof}
    This follows from the following facts.
    \begin{enumerate}
        \item $Res^{\hat{G}}_{G}(\rho)=\epsilon$.
        \item The restriction $Res^{\hat{G}}_{G}$ is a ring map.
        \item All irreducible complex representation of $\hat{G}$ are restrictions of irreducible Real $\hat{G}$-representations (by Lemma \ref{lem:RR-RU-identification}).
    \end{enumerate}
\end{proof}

\begin{remark}
    If a $G$ is equipped with the trivial augmentation and is quasi-abelian (hence abelian), then the data of a Real orientation of $E_G$ is equivalent to the data of a complex orientation of $E_G$.
\end{remark}

If $\alpha$ is an irreducible Real representation of $\eta$, we may define elements $y(\alpha) \in E_\eta^\rho(\mathbb{C}P(\mathcal{U}_\eta))$ by pulling back $y(\alpha)$ along the map $\mathbb{C}P(\mathcal{U}_\eta) \to \mathbb{C}P(\mathcal{U}_\eta)$ induced by tensoring with $\alpha$. We will see in Proposition \ref{prop:y_alpha_generate} that $E_\eta^\bigstar (\C P(\mathcal{U}_\eta))$ has additive bases given by sequences of products of the various $y(\alpha)$. 

Next, we define the notion of a Real global orientation of a Real global spectrum. Recall the notation $\widehat{\mathbb{T}} := \mathbb{T} \rtimes C_2$. As above, let $\rho$ be the Real representation of $\hat{\mathbb{T}}$ on $\mathbb{C}$, where $\mathbb{T}$ acts trivially and $C_2$ acts by complex conjugation. Let $\hat{\tau}$ be the Real representation of $\hat{\mathbb{T}}$ where $\mathbb{T}$ acts tautologically (ie by scalar multiplication), and $C_2$ acts by complex conjugation.

\begin{definition}
    Suppose $E$ be a homotopy commutative Real global spectrum. A Real global orientation of $E$ is an $RO(\widehat{\mathbb{T}})$-graded unit in $\pi_{\rho - \hat{\tau}}^{\widehat{\mathbb{T}}}(E_{\widehat{\mathbb{T}}})$ which restricts to the multiplicative identity along the canonical inclusion $C_2 \rightarrow \widehat{\mathbb{T}}$.
\end{definition}

\begin{remark}
    It would be very interesting to write down a notion of $C_p$-global orientation which generalizes the $\mu_p$-orientations of \cite{HSW20} in the same way that our definition of Real global orientation generalizes Real orientations of $C_2$-spectra. 
\end{remark}

Now we discuss how a Real global orientation of a homotopy commutative Real global ring spectrum gives rise to Real $\eta$-orientations of $E_\eta$ for each quasi-abelian $\eta$.

\begin{proposition}
\label{prop: Real global orientation induces Real hat{G} orientation}
    Let $E$ be a homotopy commutative Real global ring spectrum. Then a Real global orientation for $E$ induces a Real $\eta$-orientations on $E_\eta$ for all $\eta$.
\end{proposition}

We mimic the argument in \cite[5.4.1]{Hau22}.

\begin{proof}
    Start by taking the homotopy orbits by using the unit sphere as a model for the universal $\widehat{\mathbb{T}}$-space $E_G \widehat{\mathbb{T}}$ in $G$-spaces. Start with the tautological $\widehat{\mathbb{T}}$-representation (which has complex dimension $1$), and view it as a Real $\widehat{\mathbb{T}}$-bundle over a point. Next, restrict along the projection $p : G \times_{C_2} \widehat{\mathbb{T}} \rightarrow \widehat{\mathbb{T}}$. If we take $(\widehat{\mathbb{T}}, G)$-homotopy orbits using this model, then we obtain the tautological $G$-line bundle over the $G$-space $\C P( \mathcal{U}_\eta )$. Taking one-point compactifications, we may identify $(S^\tau)_{h_{G, \widehat{\mathbb{T}}}}$ with $\C P( \mathcal{U}_\eta )$.

    Our Real global orientation $t$ is an element of $E^{\rho}_{\widehat{\mathbb{T}}}(S^\tau)$. By restricting along the projection $p : G \times_{C_2} \widehat{\mathbb{T}} \rightarrow \widehat{\mathbb{T}},$ we obtain a unit in \[ E^{\rho}_{G \times_{C_2} \widehat{\mathbb{T}}}(p^*(S^\tau)) \] Applying homotopy orbits $h_{G,\widehat{\mathbb{T}}}(S^\tau)$ to this element yields an element $t^{(G)}$ of \[ E^{\rho}_\eta((S^\tau)_{h_{G, \widehat{\mathbb{T}}}}) \cong E^{\rho}_\eta(\C P( \mathcal{U}_\eta )) \] which we claim is a Real $\eta$-orientation.

    To show this, we must check that $t^{(G)}$ restricts to a unit along each inclusion $\C P(\epsilon \oplus \alpha) \subset \C P(\mathcal{U}_\eta)$. First, recall the identifications \[ \C P(\epsilon \oplus \alpha) \cong S^\alpha \] \[ S(\mathcal{U}_\eta)_+ \wedge_{\mathbb{T}} S^\tau \cong (S^\tau)_{h_{G, \widehat{\mathbb{T}}}} \cong \C P( \mathcal{U}_\eta ) \] and \[ S^\alpha \cong \alpha^* S^\tau \] Now the map $\alpha$ determines a graph subgroup $H$ of $G \times_{C_2} \widehat{\mathbb{T}}$, and a $H$-fixed point of $S(\mathcal{U}_\eta)$ determines a vector in $\mathcal{U}_\eta$ which determines a subrepresentation isomorphic to $\alpha$. This subrepresentation then determines a copy of $S^\alpha$ in $S(\mathcal{U}_\eta)$, which determines a map $S^\alpha \rightarrow S(\mathcal{U}_\eta)_+ \wedge_{\widehat{\mathbb{T}}} S^\tau$. This map corresponds to the inclusion of complex projective spaces via the above identifications.

    From this description, we see that the restriction of $t^{(G)}$ is equal to the image of $t$ via \[ E^{\rho}_{\widehat{\mathbb{T}}}(S^\tau) \xrightarrow{pr_{\widehat{\mathbb{T}}}^*} E^{\rho}_{G \times_{C_2} \widehat{\mathbb{T}}}(pr_{\widehat{\mathbb{T}}}^* S^\tau) \xrightarrow{h_{G, \hat{T}}(pr_{\widehat{\mathbb{T}}}^* S^\tau)} E^{\rho}_\eta((S^\tau)_{h_G \widehat{\mathbb{T}}}) \xrightarrow{(*_\alpha \wedge -)^*} E^{\rho}_\eta(S^\alpha) \] By Lemma \ref{HauLemma3.4} this is just the restriction along $\alpha$. Since restriction maps preserve multiplication, they take units to units and $1$ to $1$.
\end{proof}

\subsection{Real equivariant formal group laws}

Suppose $G$ is an abelian compact Lie group. The definition of a $G$-equivariant formal group law is modeled after $E^*_G(\C P( \mathcal{U}_G))$ for a complex oriented $G$-ring spectrum $E_G$. In this subsection, we will investigate the Real analogue of this story. 

Throughout this subsection, we let $\hat{G} = G \rtimes C_2$, and denote also by $\hat{G}$ the augmentation given by projection onto $C_2$. We will investigate the structure possessed by $E^\bigstar_{\hat{G}}(\mathbb{C}P(\mathcal{U}_{\hat{G}}))$ for a Real oriented $\hat{G}$-ring spectrum $E_{\hat{G}}$. We use $\bigstar$ to denote the $RO(C_2)$-graded cohomology of a such a spectrum $E_{\hat{G}}$, which is a sub-object of the full $RO(\hat{G})$-graded cohomology. Our first task is to prove that the $E_{\hat{G}}$-cohomology of $\mathbb{C}P(\mathcal{U}_{\hat{G}})$ is free.

Following \cite[Section 4]{CGK00}, let $\hat{V}$ and $\hat{W}$ be Real $\hat{G}$-representations. Then there is a short exact sequence
\[
0\leftarrow E_{\hat{G}}^{\bigstar}(\C P(\hat{V}))\leftarrow E_{\hat{G}}^{\bigstar}(\C P(\hat{V}\oplus \hat{W}))\leftarrow E_{\hat{G}}^{\bigstar}(\C P(\hat{V}\oplus \hat{W}),\C P(\hat{V}))\leftarrow 0.
\]
For any Real irreducible $\hat{G}$-representation $\alpha$, the element $y(\alpha)$ uniquely defines an element $x(\alpha)\in E^{\rho}_{\hat{G}}(\C P(\mathcal{U}_{\hat{G}}),\C P(\alpha))$. For $\hat{V}=\hat{\alpha_1}\oplus\cdots\oplus \hat{\alpha_n}$, we can define
\[
x(\hat{V})=x(\hat{\alpha_1})*\cdots *x(\hat{\alpha_n})\in E^{\bigstar}_{\hat{G}}(\C P(\mathcal{U}_{\hat{G}}),\C P(\hat{V})).
\]
Here $*$ is the external cup product along the map
\begin{align*}
    &\bar{\Delta}:(\C P(\hat{V}\oplus \hat{W}\oplus \hat{Z}), \C P(\hat{V}\oplus \hat{W}))\to\\
    &(\C P(\hat{V}\oplus \hat{Z}), \C P(\hat{V}))\times \bar{\Delta}:(\C P(\hat{W}\oplus \hat{Z}), \C P(\hat{W}))
\end{align*}
sending $(v:w:z)\mapsto ((v:z),(w:z))$.


\begin{proposition}
\label{prop:y_alpha_generate}
    Suppose $E_{\hat{G}}$ is a Real oriented $\hat{G}$-spectrum, and suppose $\hat{V}^n = \oplus_{i=1}^n \hat{\alpha_i}$, for $n \geq 1$, is a complete flag in a complete Real $\hat{G}$-universe $\mathcal{U}_{\hat{G}}$ with $\hat{\alpha_i}$ an irreducible Real $\hat{G}$-representation. Then $E^{\bigstar}_{\hat{G}}( \C P(\mathcal{U}_{\hat{G}}))$ is a free $E^\bigstar_{\hat{G}}$-module with basis $\{1, y(\hat{\alpha_1}), y(\hat{\alpha_1}) y(\hat{\alpha_2}),\dots\}$.
\end{proposition}

\begin{proof}
    We follow the proof of \cite[Theorem 4.3]{CGK00} (originally due to Cole \cite{Col96}), by using the cofiber sequence\footnote{Note that there is a minor misprint in the analogous cofiber sequence in \cite{CGK00}.}
    \[
    (\C P(\hat{V}^{n+1}),\C P(\hat{V}^{n}))\to (\C P(\mathcal{U}_{\hat{G}}),\C P(\hat{V}^{n}))\to (\C P(\mathcal{U}_{\hat{G}}),\C P(\hat{V}^{n+1})).
    \]
    By induction on $n$, one can show that $x(\hat{V}^n)$ (following the notation from \cite{CGK00}) restricts to a unit in the $E_{\hat{G}}$-cohomology of $(\C P(\hat{V}^{n+1}),\C P(\hat{V}^{n}))$, thus the above cofiber sequence gives rise to split short exact sequence in $E_{\hat{G}}$-cohomology.
\end{proof}

\begin{proposition}\label{Real FGL for Real-oriented spectrum}
    Let $G$ be an abelian compact Lie group and let $\hat{G} = G \rtimes C_2$. If $E_{\hat{G}}$ is a Real oriented $\hat{G}$-spectrum, then
    \[ E^\bigstar_{\hat{G}}( \C P(\mathcal{U}_{\hat{G}}) \times \C P(\mathcal{U}_{\hat{G}})) \cong E^\bigstar_{\hat{G}}(\C P(\mathcal{U}_{\hat{G}})) \widehat{\otimes}_{E^\bigstar_{\hat{G}}} E^\bigstar_{\hat{G}}(\C P(\mathcal{U}_{\hat{G}})), \]
    and we can define structure maps
    \begin{align*}
        &\hat{\Delta}: E^\bigstar_{\hat{G}}( \C P(\mathcal{U}_{\hat{G}})) \to E^\bigstar_{\hat{G}}(\C P(\mathcal{U}_{\hat{G}})) \widehat{\otimes}_{E^\bigstar_{\hat{G}}} E^\bigstar_{\hat{G}}(\C P(\mathcal{U}_{\hat{G}})),\\
        &\hat{\theta}:E^{\bigstar}_{\hat{G}}(\C P(\mathcal{U}_{\hat{G}}))\to {(E^{\bigstar}_{\hat{G}})}^{G^*}
    \end{align*}
    so that
    \[ \left( E^{\bigstar}_{\hat{G}}, E^{\bigstar}_{\hat{G}}( \C P(\mathcal{U}_{\hat{G}})),\hat{\Delta},\hat{\theta}, y(\rho) \right) \]
    is a $G$-equivariant formal group law
    whose restriction to $G\subset \hat{G}$ is the $G$-equivariant formal group law
    \[ \left( E^*_{{G}}, E^{*}_{{G}}( \C P(\mathcal{U}_{{G}})),{\Delta},{\theta}, y(\epsilon) \right). \]
\end{proposition} 

\begin{proof}
    By Proposition \ref{prop:y_alpha_generate}, we know that $E^\bigstar_{\hat{G}}(\mathbb{C}P(\mathcal{U}_{\hat{G}}))$ is a free $E^\bigstar_{\hat{G}}$-module, which implies that we have a Kunneth isomorphism \[ E^\bigstar_{\hat{G}}( \C P(\mathcal{U}_{\hat{G}}) \times \C P(\mathcal{U}_{\hat{G}})) \cong E^\bigstar_{\hat{G}}(\C P(\mathcal{U}_{\hat{G}})) \widehat{\otimes}_{E^\bigstar_{\hat{G}}} E^\bigstar_{\hat{G}}(\C P(\mathcal{U}_{\hat{G}})). \] The map $\C P(\mathcal{U}_{\hat{G}}) \times \C P(\mathcal{U}_{\hat{G}}) \rightarrow \C P(\mathcal{U}_{\hat{G}})$ classifying the tensor product of Real line bundles gives rise to the coproduct $\hat{\Delta}$. The inclusion
    \[
    G^* \cong \coprod_{\hat{\alpha}}\C P(\hat{\alpha})\to \C P(\mathcal{U}_{\hat{G}})
    \]
    is a group homomorphism, and gives rise to the map 
    \[\hat{\theta}:E^{\bigstar}_{\hat{G}}(\C P(\mathcal{U}_{\hat{G}}))\to {(E^{\bigstar}_{\hat{G}})}^{G^*}.\]
    Proof of regularity of $y(\rho)$ and the axioms for equivariant formal group laws is the same as the classical case. The second assertion follows from the fact that $\text{res}^{\hat{G}}_G E_{\hat{G}} = E_G$, $\text{res}^{\hat{G}}_G \mathcal{U}_{\hat{G}} = \mathcal{U}_G$, and $\text{res}^{\hat{G}}_G y(\rho) = y (\epsilon).$
\end{proof}

\begin{remark}
   Following \cite[Remark 11.2]{CGK00}, we can define an action of $G^*$ on $E_{\hat{G}}^\bigstar(\C P(\mathcal{U}_{\hat{G}}))$ using the formula
    \[
    l_{\alpha}(r)=(\hat{\theta}(\alpha^{-1})\otimes id)\hat{\Delta}(r).
    \]
    We can also define Euler classes by $e(\alpha)=\hat{\theta}(\rho)(y(\alpha))$.
\end{remark}

\begin{remark}\label{remark:Real FGL vs FGL}
    When $G$ is the trivial group, so that $\hat{G} = C_2$, it is well known that the $C_2$ action on the formal group law $F$ associated to a Real oriented $C_2$-spectrum corresponds to the map $x \mapsto -[-1]_Fx$, where $[-1]_Fx$ denotes the $-1$-series of $F$. This means that $C_2$ acts by group inversion on the formal group scheme $\mathbb{G} = \text{Spf}(E_{C_2}^*(\mathbb{C}P^\infty))$. We can give a similar description for any compact abelian Lie group $G$. In this more general setting, the action of $C_2$ induces the Hopf algebra antipode (the equivariant analogue of $-[-1]_Fx$) on $E_{\hat{G}}^\bigstar(\mathbb{C}P(\mathcal{U}_{\hat{G}}))$, and also on the commutative subring $E_{\hat{G}}^{\rho *}(\mathbb{C}P(\mathcal{U}_{\hat{G}})) \subset E_{\hat{G}}^\bigstar(\mathbb{C}P(\mathcal{U}_{\hat{G}}))$. This means that geometrically, $C_2$ acts by group inversion on the formal group scheme $\mathbb{G} = \text{Spf}\left(E_{\hat{G}}^{\rho *}(\mathbb{C}P(\mathcal{U}_{\hat{G}})) \right) $.
    If we restrict to the sub-group scheme $\varphi: G^* \to \mathbb{G}$, then $C_2$ similarly acts by group inversion on $G^*$, and so we may describe the $C_2$-action by the diagram 
    
     \[ \begin{tikzcd}
    G^* \arrow{r}{\varphi} \arrow{d}{\iota} & \mathbb{G} \arrow{d}{\iota} \\
    G^* \arrow{r}{\varphi} & \mathbb{G}
    \end{tikzcd} \]
    where $\iota$ denotes group inversion, and $\varphi = \text{Spf}(\hat{\theta})$. Note that \cite{CGK00} have proved that any $G$-equivariant formal group law automatically admits an antipode, this means a Real $\hat{G}$-equivariant formal group law contains the same amount of data as the underlying $G$-equivariant formal group law.
\end{remark}

\subsection{Real global group laws}

Since a Real global orientation gives Real $\eta$-orientations for each augmented quasi-abelain $\eta$, which in turn give Real $\hat{G}$-equivariant formal group laws, it is natural to expect that Real global orientations encode some ``global" algebraic structure. In this subsection, we will define just that: a Real global group law. First, we recall the notion of a global group law from \cite{Hau22}.

\begin{definition}[\cite{Hau22}]
\label{def:GGL}
    A global group law is a contravariant functor $X$ from the full subcategory of compact abelian Lie groups spanned by tori, to the category of $\mathbb{Z}$-graded rings, along with a class $e \in X(\mathbb{T})$ such that if $V : \mathbb{T}^n \rightarrow \mathbb{T}$ is a split surjection, then the sequence \[ 0 \rightarrow X(\mathbb{T}^n) \xrightarrow{V^*e} X(\mathbb{T}^n) \xrightarrow{res_{ker(V)}^{\mathbb{T}^n}} X(ker(V)) \rightarrow 0 \] is exact.
\end{definition}

\begin{theorem}[\cite{Hau22}]
    The global group law $\mathbf{L}$ defined by \[ \mathbf{L}(\mathbb{T}^n) := \pi^{\mathbb{T}^n}_*(MU_{\mathbb{T}^n}) \] is the initial object in the category of global group laws.
\end{theorem}

We focus our attention to augmented tori (which are those quasi-abelian augmented compact Lie groups of the form $\widehat{\mathbb{T}^n} := \mathbb{T}^n \rtimes C_2$, with $C_2$ acting by inversion in $\mathbb{T}^n$). Additionally note the identifications $\widehat{\mathbb{T}^n} \cong \widehat{\mathbb{T}} \times_{C_2} ... \times_{C_2} \widehat{\mathbb{T}}$. This is consitent with our usage elsewhere of $\hat{G}$ to denote an augmented group and $G$ to denote the kernel of the augmentation.

 Recall that if $E$ is a homotopy commutative Real globally oriented ring spectrum, then the $RO(C_2)$-graded homotopy groups of $E_\eta$ are contravariantly functorial in augmented maps $G \rightarrow H$. We will use the following cofiber sequence to prove a result relating these rings which is analogous to an argument of Sinha \cite{Sin01}.

\begin{proposition}
    Consider the subgroup inclusion $C_2 \rightarrow \widehat{\mathbb{T}}$ obtained by the canonical splitting  of the augmentation. Then we have a cofiber sequence \[ (\widehat{\mathbb{T}}/C_2)_+ \rightarrow S^0 \rightarrow S^{\hat{\tau}} \] of $\widehat{\mathbb{T}}$-spaces. 
\end{proposition}

\begin{proof}
    The first map sends base-point to base-point, and everything else to the other point of $S^0$. The second map is the inclusion of north and south poles. Note that $S^0$ is $\widehat{\mathbb{T}}$-homotopy equivalent to the disjoint union of a basepoint with the closed unit disk in $\C$, with its tautological $\widehat{\mathbb{T}}$-action. Our first map is then replaced by the cofibration obtained by sending the basepoint to the origin in $\C$, and $\widehat{\mathbb{T}}/C_2 \cong \mathbb{T}$ to the boundary of the disk. Quotienting out by this image clearly yields $S^{\hat{\tau}}$.
\end{proof}

\begin{definition}
\label{def:Real-GGL}
    A Real global group law is a contravariant functor $X$ from the full subcategory of augmented compact Lie groups spanned by augmented tori $\widehat{\mathbb{T}^n}$ to the category of $RO(C_2)$-graded rings along with an element $e \in X(\widehat{\mathbb{T}})$ of degree $- \rho$ such that the following condition is satisfied. For every split surjective Real character $\hat{V} : \widehat{\mathbb{T}^n} \rightarrow \widehat{\mathbb{T}}$ (which restricts to a split surjective character $V:{\mathbb{T}^n} \rightarrow {\mathbb{T}}$ having the same kernel as $\hat{V}$), define the augmented kernel as $\widehat{ker({{V}})}\coloneqq ker(\hat{V})\rtimes C_2 \cong \widehat{\mathbb{T}^{n-1}}$, then the sequence \[ 0 \rightarrow X(\widehat{\mathbb{T}^n}) \xrightarrow{{\hat{V}}^* e \cdot -} X(\widehat{\mathbb{T}^n}) \xrightarrow{res^{\widehat{\mathbb{T}^n}}_{\widehat{ker(V)}}} X(\widehat{ker(V)}) \rightarrow 0 \] is exact.
\end{definition}

\begin{remark}
    Note because of Lemma \ref{lemma:short-exact-sequence}, a Real global group laws contains more data than a global group law, as opposed to the case of equivariant formal group laws discussed in in Remark \ref{remark:Real FGL vs FGL}.
\end{remark}

\begin{proposition}\label{Real global FGL for Real globally oriented spectrum}
    If $E$ is a homotopy commutative Real globally oriented Real global ring spectrum, then the $RO(C_2)$-graded homotopy groups $E^\bigstar_{\widehat{\mathbb{T}^n}}$ assemble to a Real global group law.
\end{proposition}
\begin{proof}
    Consider the cofiber sequence
    \begin{equation}\label{canonical cofibseq of hatT}
        \widehat{\mathbb{T}}/{C_2}\to S^0\to S^{\widehat{\tau}}
    \end{equation}
    with $\widehat{\tau}$ the canonical Real $\widehat{\mathbb{T}}$-representation. Let $\hat{V}: \widehat{A}\to \widehat{\mathbb{T}}$ be a split surjective augmented map with $\widehat{A}=A\rtimes C_2$ for some compact abelian Lie group $A$. Let $V:A\to \mathbb{T}$ be the restriction of $\hat{V}$. Then $V$ is still split surjective, since if we let $\hat{s}:\widehat{\mathbb{T}}\to \widehat{A}$ to be a splitting of $\hat{V}$, we know $\hat{s}$ must be augmented, thus restricts to a splitting of $V$ along $\mathbb{T}\to \widehat{\mathbb{T}}$. We call $s=\hat{s}|_{\mathbb{T}}$. Now we have the following split short exact sequence of abelian groups
    \[\begin{tikzcd}
        0\to ker(V)\xrightarrow{i} A\xrightarrow{V} {\mathbb{T}}\to 0
    \end{tikzcd}\]
    where $i$ is the inclusion. Then we have the retraction $r:A\to ker(V)$ satisfying $r\circ i=1$. Applying $-\rtimes C_2$ gives us a retraction $\hat{r}:\widehat{A}\to \widehat{ker(V)}$ with $\hat{r}\circ\hat{i}=1$.
    
    Now pull back the cofiber sequence of \eqref{canonical cofibseq of hatT} using $\hat{V}:\hat{A}\to \widehat{\mathbb{T}}$, we get a cofiber sequence of $\hat{A}$-spaces
    \begin{equation*}
        \widehat{A}/{\widehat{ker(V)}}_+\to S^0\to S^{\hat{V}}
    \end{equation*}
    which gives rise to a long exact sequence
    \[ \cdots \rightarrow E^\bigstar_{\widehat{A}} (S^{\hat{V}}) \rightarrow E^\bigstar_{\widehat{A}}(S^0) \rightarrow E^\bigstar_{\widehat{A}}( \widehat{A}/{\widehat{ker(V)}}_+) \rightarrow\cdots. \]
    From the Real global orientation we have $E^\bigstar_{\widehat{A}} (S^{\hat{V}})\cong E^\bigstar_{\widehat{A}} (S^{\rho})$ and the first displayed map can be identified with multiplication by $e_{\hat{V}}$. The second displayed map can be identified with $res^{\widehat{A}}_{\widehat{ker(V)}}$ using $E^\bigstar_{\widehat{A}}(\widehat{A}/{\widehat{ker(V)}}_+)\cong E^\bigstar_{\widehat{ker(V)}}(S^0)$. Because of the retraction $\hat{r}$, we deduce that $res^{\widehat{A}}_{\widehat{ker(V)}}$ is surjective and we have a short exact sequence. Specializing to $A={\mathbb{T}^n}$ proves the result.
\end{proof}

\begin{construction}
    Given a Real global group law $X$, we can construct a global group law $\overline{X}$ as follows. For a torus $A \cong \mathbb{T}^n$, we define a $\mathbb{Z}$-graded ring $\overline{X}(A)$ via the formula
    \[
    \overline{X}(A)_{2n} = X(\hat{A})_{\rho n}
    \]
    \[
    \overline{X}(A)_{2n+1} = 0.
    \]
    In other words, the value of $\overline{X}$ at $A$ is the subring $\overline{X}(A) \subset X(\hat{A})$ consisting of elements of degree $n\rho$ for some $n \in \mathbb{Z}$. The euler class $e \in X(\hat{\mathbb{T}})$ has degree $-\rho$, so we have $\bar{e} \in \overline{X}(\mathbb{T})$.
\end{construction}

\begin{proposition}
\label{prop:overline-functor}
    The construction $(X,e) \mapsto (\overline{X},\bar{e})$ specifies a functor from Real global group laws to global group laws.
\end{proposition}

\begin{proof}
    First, we need to show that $(\overline{X},\bar{e})$ is in fact a global group law. Suppose $V:A \to \mathbb{T}$ is a split surjective character. Then applying $(-)\rtimes C_2$ determines a split surjective Real characther $\hat{V} : \hat{A} \to \hat{\mathbb{T}}$. The Real global group law then determines a short exact sequence 
    \[ \begin{tikzcd} 0 \ar[r] & X(\widehat{\mathbb{T}^n}) \ar[r,"\hat{V}^* e"] & X(\widehat{\mathbb{T}^n}) \ar[r,"res^{\widehat{\mathbb{T}^n}}_{\widehat{ker(V)}}"] &  X(\widehat{ker({V})}) \ar[r] & 0.
    \end{tikzcd} \] 
    Since $\widehat{ker({V})} = \text{ker}(V) \rtimes C_2$, and since restriction maps preserve $RO(C_2)$-grading, the short exact sequence above restricts to a short exact sequence on $\rho *$-graded subrings, which is of the form \[ \begin{tikzcd} 0 \ar[r] & \overline{X}(\mathbb{T}^n) \ar[r,"V^*\bar{e}"] & \overline{X}(\mathbb{T}^n) \ar[r,"res^{\mathbb{T}^n}_{ker(V)}"] &  \overline{X}(ker(V)) \ar[r] &  0. \end{tikzcd}  \] We conclude that $\overline{X}$ forms a global group law.

    That the construction $(X,e) \mapsto (\overline{X},\bar{e})$ is functorial in maps of Real global group laws follows from the fact that restriction to $\rho*$-graded subrings is functorial in maps of $RO(C_2)$-graded rings.
\end{proof}

\begin{remark}
    One might wonder whether or not there is a similar construction producing a Real global group law from a global group law, and in particular whether there is an adjoint to the above construction. After all, the data of a formal group law is equivalent to the data of a Real formal group law (that is, a formal group law together with the action of $C_2$ specified by formal inversion). As a first obstruction, we must restrict to global group laws with values in $\Z$-graded rings which are concentrated in even degrees (for example, $\underline{\pi_*}(\mathbf{MU})$ \cite{Lof73} \cite{Com96}, $\underline{\pi_*}(\mathbf{KU})$ \cite{Ati68}, and the global group law $\underline{\pi_*}(\mathbf{MU}) \otimes_{MU^*} BP^*$ alternatively constructed from $BP_A^*$ \cite{May98}, see also \cite{Wis24} for a discussion). This is because we wish to postcompose with the functor from $\Z$-graded rings concentrated in even degrees to $RO(C_2)$-graded rings given by sending $R^*$ to the $\hat{R}$ defined by $\hat{R}^{\rho n} = R^{2n}$ and $\hat{R}^{\bigstar} = 0$ otherwise.
    
    Next, observe that there is a functor from the full subcategory of quasi-abelian augmented groups spanned by augmented tori to the full subcategory of abelian groups spanned by tori, which is a left inverse to the functor $\mathbb{T}^n \mapsto \widehat{\mathbb{T}^n} := \mathbb{T}^n \rtimes C_2$. This functor is given by $\widehat{\mathbb{T}^n} \mapsto Ker(\widehat{\mathbb{T}^n} \rightarrow C_2)$. Additionally, there is a functor from $\Z$-graded rings concentrated in even degrees to $RO(C_2)$-graded rings given by sending $R^*$ to the ring $R^{\bigstar}$ defined by $R^{* \rho} = R^{2*}$ and $R^{\bigstar} = 0$ otherwise. Given a global group law with values in graded rings which are concentrated in even degrees, we obtain a Real global group law by respectively precomposing and postcomposing with the two functors described above. The resulting functor is denoted $Y \mapsto \overline{Y}$.

    In fact it does not seem likely that $X \mapsto \overline{X}$ is left adjoint to $Y \mapsto \overline{Y}$, essentially due to the failure of $\mathbb{T}^n \mapsto \widehat{\mathbb{T}^n}$ to be full. In particular, it seems that Real global group laws actually contain much more information than global group laws. This greatly contrasts the ordinary case, for which the formal group law data arising from a Real-oriented $C_2$-spectrum is essentially the same as the data arising from a complex oriented spectrum (modulo $RO(C_2)$ versus $\Z$ grading issues).
\end{remark}

Working $2$-locally, one might expect the existence of a homotopy commutative Real global spectrum $BP \mathbb{R}$, so that whenever $\eta : G \rightarrow C_2$ is trivial, the resulting $G$-spectrum $BP \mathbb{R}_\eta$ is just the $BP_G$ considered in \cite{May98} and \cite{Wis24}. Indeed, it seems likely that such an object exists, and possesses a Real global orientation. However, a serious treatment of such a matter would take us too far afield.

\section{Construction and properties of $\mathbf{MR}$}

In this section we introduce the $C_2$-global spectrum $\mathbf{MR}$ and its periodic version $\mathbf{MRP}$. We show that $\mathbf{MR}$ is canonically Real globally oriented following \cite{Sch14}.

\subsection{Construction of $\mathbf{MR}$}

The unitary spectrum models for $\mathbf{MR}$ and $\mathbf{MRP}$ are constructed in \cite{Sch14}. To avoid confusion, we will write $\mathbf{MR}^U$ and $\mathbf{MRP}^U$ for the unitary versions (where $\mathbf{MR}^U$ is Schwede's $\mathbf{MR}^{[0]}$ and $\mathbf{MRP}^U$ is his $\mathbf{MR}$). There is an underlying functor $u:Sp^U\to Sp$ from unitary spectra to orthogonal spectra, which actually lands in $Sp^{C_2}$ where the $C_2$-action comes from complex conjugation. 

More concretely, we define orthogonal $C_2$-spectra $\mathbf{MR}$ and $\mathbf{MRP}$ as follows, which then can be regarded as objects in $\mathcal{GH}_{C_2}$. For $V$ an inner product space over the real numbers, let $V_{\C}=V\otimes_{\R}\C$ to be its complexification. $V_{\C}$ has a natural $C_2$-action where it acts on the $\C$ factor by complex conjugation. 

Let
\[ \mathbf{BR}(V) =  Gr_{dim_\C (V_{\C})}(V_{\C}^2) \]
be the complex Grassmanian, where $V_{\C}^2=V_{\C}\oplus V_{\C}$. It has a $C_2$-action from complex conjugation on $V_{\C}$. Over this space there is the tautological hermitian vector bundle and we use $Th(\mathbf{BR}(V))$ to denote its Thom space. We define
\[
\mathbf{MR}(V)=Map_*(S^{iV},Th(\mathbf{BR}(V)))=\Omega^{iV}Th(\mathbf{BR}(V)).
\]
Here $C_2$ acts on $iV$ by $-1$ and acts on the entire function space by conjugation action on functions. Now let $W$ be another inner product space. We have to define $C_2$-equivariant maps
\[
S^W\wedge \mathbf{MR}(V)\to \mathbf{MR}(V\oplus W)
\]
By adjunction, it suffices to define a map
\[
S^{W_{\C}}\wedge Th(\mathbf{BR}(V))\to Th(\mathbf{BR}(V\oplus W)).
\]
This is defined by $w\wedge (U,x)\mapsto (\kappa_{V,W}(U\oplus W_{\C}\oplus 0),\kappa_{V,W}(x,w,0))$, where $\kappa_{V,W}:V_{\C}^2\oplus W_{\C}^2\to (V\oplus W)_{\C}^2$ is the linear isometry given by $\kappa_{V,W}(v,v',w,w')=(v,w,v',w')$. 

$\mathbf{MR}$ can be equipped with $C_2$-equivariant maps $\mathbf{MR}(V)\wedge \mathbf{MR}(W)\to \mathbf{MR}(V\oplus W)$ and $S^V\to \mathbf{MR}(V)$ making it an commutative ring in orthogonal $C_2$-spectra.

\begin{definition}\label{def: MR}
    The orthogonal $C_2$-spectrum $\mathbf{MR}$ is the one defined above. It has the structure of a commutative ring in orthogonal $C_2$-spectra. We also use $\mathbf{MR}$ to denote its image in $\mathcal{GH}_{C_2}$.
\end{definition}

The above construction actually produces the $0$-th summand of a family of spectra $\mathbf{MR}^{[n]}$ for any $n\in \mathbb{Z}$. Let $V$ be an inner product space, we define the complex Grassmanian
\[ \mathbf{BR}^{[n]}(V) =  Gr_{dim_\C (V_{\C})+n}(V_{\C}^2) \] and we use $Th(\mathbf{BR}^{[n]}(V))$ to denote the Thom space of the tautological bundle of $(dim_{\C}V_{\C}+n)$-dimensional planes over it. Similar structure maps can be defined as in the $n=0$ case, and we get an orthogonal $C_2$-spectrum $\mathbf{MR}^{[n]}$. We define
\[ \mathbf{MRP} = \bigvee_{n \in \Z} \mathbf{MR}^{ \left[ n \right] } \]
which can be equipped with the structure of a commutative ring in orthogonal $C_2$-spectrum.

\begin{remark}
    The underlying orthogonal $C_2$-spectrum of Schwede's unitary $\mathbf{MRP}^U$ in \cite{Sch14} is given by
    \[
    \mathbf{MRP}^U(V)=\Omega^{iV}(\bigvee_{n\in\mathbb{Z}}Th(\mathbf{BR}^{[n]}(V))).
    \]
    This is $C_2$-globaly equivalent but not equal to the $\mathbf{MRP}$ defined above.
\end{remark}

\begin{definition}\label{def: MR_eta}
Let $\eta:G\to C_2$ be an augmented compact abelian Lie group. Let $\eta^*:Sp^{C_2}\to Sp^G$ be the pullback functor giving an orthogonal $C_2$-spectrum a $G$-action via $\eta$. We define
\[
M\R_{\eta}=\eta^*(\mathbf{MR})\text{ and }M\R P_{\eta}=\eta^*(\mathbf{MRP}).
\]
\end{definition}

The point-set level pullback functor $\eta^*:Sp^{C_2}\to Sp^G$ is strong symmetric monoidal, since the smash product in both categories are defined by first taking the non-equivariant smash product and give it the diagonal actions. Thus $M\R_{\eta}$ and $M\R P_{\eta}$ automatically become commutative rings in $Sp^G$. We summarize as follows
\begin{proposition}
\label{prop:MRP_ring_structure_restricts}
	Let $\eta : G \rightarrow C_2$ be an augmented compact Lie group. The commutative ring structure on $\mathbf{MR}$ induces commutative ring structure on the orthogonal $G$-spectra $M \mathbb{R}_{\eta}$. These commutative ring structures are compatible with augmented group homomorphisms in the following sense: Let $f:H\to G$ be an augmented map. Then $Res_f(M\R_{\eta})=M\R_{\eta\circ f}$ as commutative ring in orthogonal $H$-spectra. Similar results are true for $\mathbf{MRP}$ and $M\R P_{\eta}$.
\end{proposition}

The $G$-homotopy groups of a unitary spectrum are isomorphic to the $G$-homotopy groups of the underlying orthogonal $C_2$-spectrum via a cofinalty argument. Thus, as a corollary of \cite[Propostion 7.8]{Sch14}, we get a $C_2$-global equivalence $\mathbf{MR}^{[-n]}\simeq S^{n\rho}\wedge\mathbf{MR}^{[0]}$, which implies the following.

\begin{proposition}\label{prop:MRPeta splits}
	For any augmented compact Lie group $\eta : G \rightarrow C_2$, we have a canonical equivalence $M \mathbb{R} P_{\eta} \simeq \vee_{n \in \mathbb{Z}} S^{n \rho} \wedge M \mathbb{R}_{\eta}$.
\end{proposition}

\begin{remark}
\label{rem:weak-equivalences-are-the-same}
    The cofinalty argument mentioned in the last paragraph goes as follows. Let $\eta:G\to C_2$. In \cite{Sch14}, Schwede defined the $G$-homotopy groups of $\mathbf{MR}^U$ to be
    \[ \pi^G_0(\mathbf{MR}^U) := \textrm{colim}_{\hat{V} \in \mathcal{U}_{\eta}} [S^{\hat{V}},\mathbf{MR}^U(\hat{V})]^G \]
    where $\mathcal{U}_{\eta}$ is a complete Real $\eta$-universe. This is isomorphic to the $G$-homotopy group of the $G$-spectrum $M\mathbb{R}_{\eta}$ (thus $\pi^G_0(\eta^* \mathbf{MR})$) as follows. By definition, let $\mathcal{U}_G$ be a complete universe for orthogonal $G$-representations, then
    \begin{align*}
    \pi^G_0(M\mathbb{R}_{\eta}) :=& \textrm{colim}_{W \in \mathcal{U}_G} [S^{W},M\mathbb{R}_{\eta}(W)]^G\\
    \cong &\textrm{colim}_{W \in \mathcal{U}_G} [S^{W\otimes_{\R}\C},\mathbf{MR}^U(W\otimes_{\R}\C)].
    \end{align*}
    To see that this two groups are isomorphic, we only need to show that the following association $\varphi$ from orthogonal $G$-representations to Real $\eta$-representations is cofinal. Let $V$ be an orthogonal $G$-representation, define
    \[
    \varphi(V)=V\otimes_{\R}\C.
    \]
    For $g\in G, v\in V$ and $z\in \C$, let $g\cdot (v,z)=(g\cdot v,\eta(g)\cdot z)$, where we uses the $C_2\cong Gal(\C/\R)$-action on $\C$. This makes $\varphi(V)$ into a Real $\eta$-representation. We prove that for any actual Real $\eta$-representation $\hat{V}$, it embeds into $\varphi(W)$ for some actual orthogonal $G$-representation $W$. 
    
    Let $V$ be the underlying complex $Ker(\eta)$-representation of $\hat{V}$, and $un(\hat{V})$ be the underlying orthogonal $G$-representation. Then we have
    \[
    res^{G}_{Ker(\eta)}(\varphi(un(\hat{V})))\cong V\oplus \bar{V}
    \]
    where $\bar{V}$ is the conjugate representation. Then we have $V=res^{G}_{Ker(\eta)}(\hat{V})\hookrightarrow res^{G}_{Ker(\eta)}(\varphi(un(\hat{V})))$. Applying $tr^{G}_{Ker(\eta)}$ on both sides gives $2\hat{V}\hookrightarrow 2\varphi(un(\hat{V}))$. So we can just take $W=un(\hat{V})$.
\end{remark}

From the compatibility of the commutative ring structures under pullbacks, we see that these restriction maps impart the homotopy groups \[ \pi^{\widehat{\mathbb{T}}^n}_*(M \mathbb{R} P_{\widehat{\mathbb{T}}^n}) \] and \[ \pi^{\widehat{\mathbb{T}}^n}_*(M \mathbb{R}_{\widehat{\mathbb{T}}^n}) \] with the structure of a contravariant functor from the full subcategory of augmented Lie groups spanned by the augmented tori to $\mathbb{Z}$-graded rings.

\begin{lemma}
    Let $\eta: G \rightarrow C_2$ be an augmented compact Lie group. Then the underlying orthogonal $ker(\eta)$-spectrum $M \R_{\eta\circ i}$ ($i:ker(\eta)\to G$ is the inclusion) is $MU_{ker(\eta)}$ as commutative rings. In particular, if $\eta$ is the trivial augmentation, then $M\R_{\eta}$ is the orthogonal $G$-spectrum $MU_G$ in \cite{Sch18}.
\end{lemma}

Actually, we have the commutative diagram
\[\begin{tikzcd}
	{Sp^{C_2}} && {Sp^G} \\
	\\
	Sp && {Sp^{ker(\eta)}}.
	\arrow["{\eta^*}", from=1-1, to=1-3]
	\arrow["{res^{C_2}_e}", from=1-1, to=3-1]
	\arrow["{res^G_{ker(\eta)}}", from=1-3, to=3-3]
	\arrow["triv", from=3-1, to=3-3]
\end{tikzcd}\]
If the two categories on the left are equipped with the $C_2$-global and global model structures respectively, and the two categories on the right are equipped with the genuine model structures, all functors in this diagram are homotopical thus are already derived. $\mathbf{MR}$ restricts to Schwede's $\mathbf{MU}$ by definition, which pullback to $MU_{ker{\eta}}$ via the trivial action functor.

\begin{proposition}\label{prop: Real global orientation restricts to complex orientation}
    Suppose $E$ be a homotopy commutative Real global spectrum. A Real global orientation of $E$ restricts to a complex orientation of $u(E)\in \mathcal{GH}$, where $u$ is the derived functor of $res^{C_2}_e$.
\end{proposition}
\begin{proof}
    Simple diagram chase by restriction from $\hat{\mathbb{T}}$ to the subgroups $\mathbb{T},C_2$ and $e$.
\end{proof}

Now we wish to extend our homotopy groups from $\Z$-graded to $RO(C_2)$-graded.

\begin{definition}
	Let $\alpha \in RO(C_2)$ and $\eta : G \rightarrow C_2$. We define the $RO(C_2)$-grading on $\pi^G_*(M \mathbb{R}_\eta)$ by \[ \pi^G_\alpha(M \mathbb{R}_\eta) := \pi^G_{\eta^* \alpha}(M \mathbb{R}_\eta). \] We use $\bigstar$ to denote this $RO(C_2)$-grading.
\end{definition}

Observe that the ring structure on $M \mathbb{R}_\eta$ makes $\pi^G_{\bigstar}(M \mathbb{R}_\eta)$ into an $RO(C_2)$-graded ring. The following is immediate from the definitions.

\begin{proposition}
	Let $f : H \rightarrow G$ be any morphism of augmented groups. Then the restriction $f_* : \pi^G_* (M \mathbb{R}_\eta) \rightarrow \pi^H_* (M \mathbb{R}_{\eta \circ f})$ canonically extends to a morphism of $RO(C_2)$-graded rings.
\end{proposition}

\subsection{Orientability of $\mathbf{MR}$}

Let $\eta : G \rightarrow C_2$ be an augmented compact Lie group and $V$ a Real $\eta$-representation. Let $dim_{\C}V=n$, from the definition, we have an homeomorphism \[ S^{V^2} \cong Th(\mathbf{BR}^{ \left[ n \right] }(V)) \] which defines an element $\sigma_{\eta,V} \in \pi^G_{V} (M \mathbb{R} P_\eta)\cong \pi^G_{V}(\mathbf{MRP}^U)$. The following result proved in \cite{Sch14} shows that $\mathbf{MR}$ posses a Real global orientation $\sigma_{\hat{\mathbb{T}}, \hat{\tau}}$.

\begin{proposition}{\cite[Proposition 7.8]{Sch14}}
\label{PropertiesOfOrientationClasses}
    Let $\eta:G\to C_2$ and $\epsilon:K\to C_2$ be augmented compact Lie groups.
    \begin{enumerate}
        \item The suspension isomorphism on homology induced by an isomorphism $V \cong W$ of Real $\eta$-representations sends $\sigma_{\eta,W}$ to $\sigma_{\epsilon,V}$.
        \item $\sigma_{\eta,0}$ is the multiplicative unit.
        \item If $\alpha : \epsilon \rightarrow \eta$ is a homomorphism of augmented Lie groups, then $\alpha^*(\sigma_{\eta,V}) = \sigma_{\epsilon,\alpha^*(V)}$.
        \item For a Real $\eta$-representation $V$ and Real $\epsilon$-representation $W$, we have $\sigma_{\eta,V} \cdot \sigma_{\epsilon,W} = \sigma_{\eta \times_{C_2} \epsilon, V \oplus W}$.
        \item For every $k\in\mathbb{Z}$, \[ (-) \cdot \sigma_{\eta,V} : \pi_k^G (\mathbf{MRP}^U) \rightarrow \pi^G_{k+V}( \mathbf{MRP}^U) \] is an isomorphism.
    \end{enumerate}
\end{proposition}

\begin{proof}
    We borrow the proof for $\mathbf{MRP}^U$ from \cite{Sch14}. The $C_2$-global version follows because they share the same homotopy groups. 1, 2, 3, and 4 are straightforward using the adjunction between $S^V \wedge -$ and $\Omega^V ( - )$. For 5, given a map $f : S^{k+V+U} \rightarrow \mathbf{MRP}^U(U)$ representing an element of $\pi^G_{k+V}(\mathbf{MRP}^U)$, we first construct a new map \[ S^{(k+U+V)} \xrightarrow{f \wedge (\{0\},0)} \mathbf{MRP}^U(U) \wedge \mathbf{MRP}^U(V) \xrightarrow{\mu_{U,V}} \mathbf{MRP}^U(U\oplus V). \] 
    Here the map $(\{0\},0)$ is defined as follows. Notice if $dim_{\C}V=n$, $({\mathbf{MR}^U})^{[-n]}(V)=S^0$ and we get a map given by inclusion
    \[
    (\{0\},0):S^0\to \mathbf{MRP}^U(V).
    \]
    This defines a map $F : \pi^G_{k+V}(\mathbf{MRP}^U) \rightarrow \pi_k^G (\mathbf{MRP}^U)$ which we will observe is inverse to multiplication by $\sigma_{\eta,V}$. 

    Observe that $F(-) \cdot \sigma_{\eta,V}$ is given by \[ S^{V^2} \rightarrow \mathbf{MR}^U(V^2)\]
    \[ (v,v') \mapsto (V \oplus 0 \oplus V \oplus 0, (v,0,v',0). \] There is a $G$-equivariant homotopy from $(v,v') \mapsto (v,0,v',0)$ to $(v,v',0,0)$, which induces a homotopy from $F(-) \cdot \sigma_{\eta,V}$ to the unit map $S^{V^2} \rightarrow \mathbf{MR}^U(V^2)$. Upon passing to $G$-equivariant homotopy groups by stabilizing, we obtain the identity map. A similar analysis shows that $F(- \cdot \sigma_{ \eta, V})$ is the identity map as well.
\end{proof}

\begin{corollary}
    \label{cor:MR-Real-orid}
    $\mathbf{MR}$ is Real globally oriented, ie there is an $RO(\widehat{\mathbb{T}})$-graded unit in $\pi_{\rho - \hat{\tau}}^{\widehat{\mathbb{T}}}( M \mathbb{R}_{\widehat{\mathbb{T}}} )$ which restricts to the multiplicative identity along the canonical spliting $C_2 \rightarrow \widehat{\mathbb{T}}$.
\end{corollary}

\begin{proof}
Apply Proposition \ref{PropertiesOfOrientationClasses} with $\eta :\hat{\mathbb{T}}\to C_2$ the projection and $V = \hat{\tau}$, the canonical representation of $\mathbb{\hat{T}}$. Note that $\sigma_{\hat{\mathbb{T}},\hat{\tau}} \in \mathbf{MR}^{\left[ 1 \right]}(S^{\hat{\tau}})$ and $\mathbf{MR}^{\left[ 1 \right] } \simeq S^{\rho}\wedge \mathbf{MR}$. Consequently we may regard $\sigma_{\hat{\mathbb{T}},\hat{\tau}}$ as an element of $\pi_{\hat{\tau}-\rho}^{\hat{\mathbb{T}}}(\mathbf{MR})$. Then (2) and (3) show that $\sigma_{\hat{\mathbb{T}},\hat{\tau}}$ restricts to the multiplicative identity, and (5) shows that $\sigma_{\hat{\mathbb{T}},\hat{\tau}}$ is an $RO(\hat{\mathbb{T}})$-graded unit.
\end{proof}

Now by Proposition \ref{Real global FGL for Real globally oriented spectrum} we obtain a Real global group law, and thus introduce one of our key players: $\mathbf{R}$.

\begin{corollary}
    The $RO(C_2)$-graded homotopy groups $\pi_{\bigstar}^{\widehat{\mathbb{T}^n}}(M \R_{\widehat{\mathbb{T}^n}})$ assemble to a Real global group law $\underline{\pi}_{\bigstar}(\mathbf{MR})$.
\end{corollary}

\begin{definition}
\label{def:MR-gives-a-GGL}
    Denote by $\mathbf{R}$ the global group law obtained by applying the functor of Proposition \ref{prop:overline-functor} to the Real global group law $\underline{\pi}_{\bigstar}(\mathbf{MR})$.
\end{definition}

On the other hand, using the machinery of our induction isomorphisms we obtain the following.

\begin{corollary}
    Let $\eta : G \rightarrow C_2$ be a quasi-abelian augmented group. Then the $G$-spectrum $M \R_\eta$ is Real $\eta$-oriented.
\end{corollary}

\begin{proof}
    Combine Corollary \ref{cor:MR-Real-orid} with Proposition \ref{prop: Real global orientation induces Real hat{G} orientation}.
\end{proof}

\section{Hu-Kriz Maps}

Now that we have constructed $\mathbf{MR}$ and established its basic properties, we may begin to work towards our main theorem. 

\begin{proposition}
\label{prop:res-is-a-GGL-map}
    The ring maps \[ Res^{\widehat{\mathbb{T}}^n}_{\mathbb{T}^n} : \pi^{\widehat{\mathbb{T}}^n}_{*\rho}(M \mathbb{R}_{\widehat{\mathbb{T}}^n}) \rightarrow \pi^{\mathbb{T}^n}_*(MU_{\mathbb{T}^n}) \] assemble to a map of global group laws $\mathbf{R} \rightarrow \mathbf{L}$.
\end{proposition}

\begin{proof}
    By functoriality of restrictions for genuine $G$-spectra, we see that $Res^{\widehat{\mathbb{T}}^n}_{\mathbb{T}^n}$ is a natural transformation of contravariant functors from $tori$ to rings. Note that we are using the fact that the global group law structure on $\mathbf{R}$ arises from the restriction maps between $\pi^{\widehat{\mathbb{T}}^n}_*(M \mathbb{R}_{\widehat{\mathbb{T}}^n})$ as in Definition \ref{def:MR-gives-a-GGL}. 
\end{proof}

The following theorem does not depend on any conjectural assumptions. Recall that since $\mathbf{L}$ is the initial global group law, there is a unique global group law map $\mathbf{L} \rightarrow \mathbf{R}$.

\begin{theorem}
\label{thm:split-surj}
    Restriction to $\mathbb{T}^n \subset \widehat{\mathbb{T}}^n$ defines a split surjection \[ \pi^{\widehat{\mathbb{T}}^n}_{\rho *} ( M \mathbb{R}_{\widehat{\mathbb{T}}^n}) \twoheadrightarrow \pi^{\mathbb{T}^n}_{2*}(MU_{\mathbb{T}^n}) \textrm{.} \]
\end{theorem}

\begin{proof}
    Observe that we have the composition \[ \mathbf{L} \rightarrow \mathbf{R} \rightarrow \mathbf{L} \] of morphisms of global group laws. Since $\mathbf{L}$ is initial, this composition is the identity, hence $\mathbf{R}(\mathbb{T}^n) \rightarrow \mathbf{L}(\mathbb{T}^n)$ is surjective. It is split by the unique map of global group laws $\mathbf{L} \rightarrow \mathbf{R}$.
\end{proof}

Using Real equivariant formal group laws, we can actually prove a stronger version of the above theorem.

\begin{theorem}
    Let $G$ be an abelian compact Lie group and let $\hat{G} = G \rtimes C_2$. Then the restriction map along $G \subset \hat{G}$ defines a split surjection \[ \pi^{\hat{G}}_{\rho *} ( M \mathbb{R}_{\hat{G}}) \twoheadrightarrow \pi^{G}_{2*}(MU_G) \textrm{.} \]
\end{theorem}
\begin{proof}
    $M\R_{\hat{G}}$ admit a Real $\eta$-orientation from the Real global orientation of $\mathbf{MR}$ by Proposition \ref{prop: Real global orientation induces Real hat{G} orientation}. It restricts to the canonical equivariant orientation of $MU_G$ by Proposition \ref{prop: Real global orientation restricts to complex orientation}. Thus it carries a $G$-equivariant formal group law which restricts to the universal one on $\pi_*^G MU_G$ by Proposition \ref{Real FGL for Real-oriented spectrum}. Let $\varphi: \pi^G_* MU_G\to \pi^{\hat{G}}_{*\rho}M\R_{\hat{G}}$ be the classifying map, then it is a right inverse of the restriction map.
\end{proof}

We will give two proofs that the restriction maps $\mathbf{R}(\mathbb{T}^n) \rightarrow \mathbf{L}(\mathbb{T}^n)$ are injective. Both will require a conjectural assumption, which we plan to prove in future work. That the Evenness Conjecture implies the Regularity Conjecture will be proved in Section \ref{sec: Evenness Conj implies Regularity Conj}.

\subsection{Proof using an evenness conjecture}

\begin{conjecture}[Evenness Conjecture]
\label{conj:MR-eta-is-even}
    For any compact augmented quasi-abelian Lie group $\eta : G \rightarrow C_2$, we have \[ \pi^G_{* \rho - 1}( M \mathbb{R}_\eta ) = 0 \textrm{.} \]
\end{conjecture}

Note that if $\eta$ is trivial, then $G$ is abelian, and this is precisely the statement that $MU_G^*$ is concentrated in even degrees (which is known to be true). Additionally, this kind of evenness has been observed to be a useful notion of evenness for $C_2$-equivariant homotopy theory, particularly in relation to $M \mathbb{R}$, cf \cite{HM15}.

\begin{theorem}\label{thm: Hu-Kriz iso}
    Assuming Conjecture \ref{conj:MR-eta-is-even}, the canonical map $\mathbf{L} \rightarrow \mathbf{R}$ is an isomorphism with inverse the restriction map of Proposition \ref{prop:res-is-a-GGL-map}. More generally, the restriction \[ \pi^{\hat{G}}_{\rho *} ( M \mathbb{R}_{\hat{G}}) \twoheadrightarrow \pi^{G}_{2*}(MU_G)\]
    is an isomorphism for all compact abelian Lie groups $G$.
\end{theorem}

\begin{proof}
    It suffices by Theorem \ref{thm:split-surj} show each map $\mathbf{R}(\mathbb{T}^n) \rightarrow \mathbf{L}(\mathbb{T}^n)$ is injective. Let $G=\widehat{\mathbb{T}}^n$ with the canonical augmentation. From Proposition \ref{prop:MRPeta splits}, we have $\pi_{-1}^G M\R P_G\cong \pi_{*\rho-1}^G M\R_{G}$. Then Conjecture \ref{conj:MR-eta-is-even} implies that $\pi_{-1}^G M\R P_G$=0. Since $\pi_{\bigstar}^G M\R P_G$ is $\rho$-periodic, we have $\pi_{\sigma}^G M\R P_G=0$, which implies that $\pi_{*\rho+\sigma}^G M\R_G=0$. We have the following cofiber sequence
    \[
    {G/{\mathbb{T}}^n}_+\to S^0\to S^{\sigma},
    \]
    which leads to a long exact sequence
    \begin{equation*}
    \cdots\to \pi^G_{*\rho+\sigma}M\R_{G}\xrightarrow{a_{\sigma}} \pi^G_{*\rho}M\R_{G} \xrightarrow{res^G_{\mathbb{T}^n}}\pi^{\mathbb{T}^n}_{2*}MU_{\mathbb{T}^n}\to\cdots   
    \end{equation*}
    where $a_\sigma$ is the Euler class associated to the $C_2$-representation $\sigma$. Since the leftmost group is zero, we deduce that the map $\mathbf{R}(\mathbb{T}^n)\to \mathbf{L}(\mathbb{T}^n)$ is injective. The general case is proved in the same way, as the above cofiber sequence holds for general compact abelian Lie groups, not just for tori.
\end{proof}

\subsection{Proof in the case of augmented-tori using a regularity conjecture}

Alternatively, one might hope to continue to mimic the arguments of \cite{Hau22}. In particular, the main result of \cite{Hau22} is that a certain map of global group laws is an isomorphism. That this map is injective is deduced with the help of regularity of the global group laws in question.

\begin{conjecture}[Regularity Conjecture]
\label{conj:R-is-regular}
    The global group law $\mathbf{R}$ is regular in the sense of \cite[Lemma 5.8]{Hau22}.
\end{conjecture}

For simplicity, we denote the Real $\hat{G}$-representation \[\infty\bar{\hat{\rho}}=\underset{\hat{\alpha} \in Irr(\hat{G})}{\oplus}\infty \hat{\alpha} \]
where $Irr^*(\hat{G})$ denotes the countable set of non-trivial irreducible Real representations of $\hat{G}$ and $\infty \hat{\alpha}$ denotes an countable direct sum of copies of $\hat{\alpha}$. Let $Irr^*(G)$ be the countable set of non-trivial irreducible complex representations of $G$. Let $\infty\bar{\hat{\rho}}'=\underset{{\alpha} \in Irr({G})}{\oplus}\infty \hat{\alpha}$ where $\hat{\alpha}$ is a fixed chosen lift of $\alpha$ along $res:RR(\eta)\to RU(G)$. By Lemma \ref{lem:RR-RU-identification} we obviously have $S^{\infty\bar{\hat{\rho}}}\simeq S^{\infty\bar{\hat{\rho}}'}$.

\begin{lemma}\label{lem G-geo fixed point of hatG spectrum}
    Let $G=\mathbb{T}^n$ and $\hat{G}=G\rtimes C_2$. Then we have
    \[
    S^{\infty\bar{\hat{\rho}}'}\simeq S^{\infty\bar{\hat{\rho}}} \simeq \widetilde{E\mathcal{F}[G]}
    \]
    and thus $(E\wedge S^{\infty\bar{\hat{\rho}}})^G\simeq \Phi^G(E)$ for a $\hat{G}$-spectrum $E$. Here $\mathcal{F}[G]$ is the family of subgroups of $\hat{G}$ that does not contain $G$.
\end{lemma}

\begin{proof}
        Let $\hat{H}\subset \hat{G}$ be a closed subgroup. Clearly, if $G\subset \hat{H}$, then $(\infty\bar{\hat{\rho}})^{\hat{H}}=0$ and $(S^{\infty\bar{\hat{\rho}}})^{\hat{H}}=S^0$. Now assume $G\not\subset \hat{H}$, we want to show that $(\bar{\hat{\rho}})^{\hat{H}}\neq 0$, so that $(S^{\infty\bar{\hat{\rho}}})^{\hat{H}}\simeq *$, and the claim follows.
        
        If $\hat{H}$ is trivially augmented, then we have $\hat{H}=H\subsetneq G$ and any non-trivial irreducible complex $G/H$-representation gives rise to a Real $\hat{G}$-representation (see Lemma \ref{lem:RR-RU-identification}) $\hat{V}$ such that $(\hat{V})^{\hat{H}}=\C \neq 0$.
        
        Now we can assume $\hat{H}\to C_2$ is surjective. Let $H=\{h \in G|(h,1)\in\hat{H}\}$, then $H\subset G$ is a proper subgroup. If $(1,\gamma)\in \hat{H}$, then $\hat{H}=H\rtimes C_2$, and any non-trivial irreducible complex $G/H$-representation gives rise to a Real $\hat{G}$-representation with non-zero $\hat{H}$-fixed point as above. If $(1,\gamma)\notin \hat{H}$, choose an odd element $(g_0,\gamma)\in\hat{H}$. Then we have $\hat{H}=H\coprod H(g_0,\gamma)$. Let $W$ be any non-trivial irreducible complex $G/H$-representation regarded as a $G$-representation. Then $(g_0,\gamma)$ acts as an order two automorphism of $W$ which is conjugation linear, i.e., it is a Real structure on $W$. Then $W^{\hat{H}}\cong\R\neq 0$ and we are done.
\end{proof}

\begin{proposition}
\label{prop:algebraic-gfp-are-same-as-topological-gfp}
    The geometric fixed points $\Phi^{\mathbb{T}^n} \mathbf{R}$ defined in \cite{Hau22} as $\mathbf{R}(\mathbb{T}^n)[(V^*e^{-1})]$ where $V$ runs through all surjections $V : \mathbb{T}^n \rightarrow \mathbb{T}$ (not necessarily split) may be computed as \[ \pi^{C_2}_{\rho *}(\Phi^{\mathbb{T}^n} M \mathbb{R}_{\widehat{\mathbb{T}}^n}) \] where $\Phi^{\mathbb{T}^n}$ denotes the usual $\mathbb{T}^n$-geometric fixed points of a $\widehat{\mathbb{T}}^n$-spectrum, equipped with its residual $C_2 = W_{\widehat{\mathbb{T}}^n} \mathbb{T}^n$-action.
\end{proposition}

\begin{proof}
    This follows from Lemma \ref{lem G-geo fixed point of hatG spectrum}.
\end{proof}

\begin{theorem}
\label{thm:res-induces-iso-on-geom-fixed-points}
    Let $\hat{G} = G \rtimes C_2$ for $G$ any compact abelian Lie group. Then the restriction map $\Phi^G(res^{\widehat{G}}_{G})$ given by 
    \[\pi_{*\rho}^{C_2}\Phi^G(M\R_{\hat{G}})\cong [S^{*\rho},M\R_{\hat{G}}\wedge \widetilde{E\mathcal{F}[G]}]^{\hat{G}}\to [S^{2*},MU_G\wedge \widetilde{E\mathcal{P}}]^G\cong \pi_{2*}\Phi^G MU_G\] is an isomorphism.
\end{theorem}

\begin{proof}
     We will follow the geometric fixed points computation of Sinha \cite{Sin01} (although the original computation was done by tom Dieck \cite{tom70}). 
    
     Write $\xi_n^{\hat{G}}$ for the canonical Real $\hat{G}$-bundle of complex dimension $n$. Note that the proof of \cite[Lemma 4.7]{Sin01} goes through to compute the $G$ fixed point space of the canonical Real $\hat{G}$-bundle of complex dimension $n$. This fixed point space possesses a residual $C_2$-action (the Weyl group action), which is explicitly given by \[ \bigvee_{W \in R\R^+(\hat{G})_n} T(\xi^{\hat{G}}_{|W^{G}|}) \wedge \left( \prod_{V \in Irr^*(\hat{G})} B \R (\nu_V(W)) \right) \] Here, $B \R$ denotes $BU$ with the $C_2$-action of complex conjugation, and $C_2$ acts on the Thom spaces $T(\xi^{\hat{G}}_{|W^{G}|})$ through the one-point compactification of the $C_2$-action on $\xi^{\hat{G}}_{|W^{G}|}$. Additionally, $R\R^+(\hat{G})_n$ denotes the set of complex dimension $n$ Real $\hat{G}$-representations, $Irr^*(\hat{G})$ denotes the set of nontrivial irreducible Real $\hat{G}$-representations, and $\nu_V(W)$ is the largest integer such that $\oplus_{\nu_V(W)} V$ is a summand of $W$.

     Now note that $M \R_{\hat{G}}$ is the colimit of the (suspension $G$-spectra of the) Thom spaces of the bundles $\xi_n^{\hat{G}}$. Geometric fixed points commute with this colimit and are computed on suspension spectra by actual fixed points, hence the preceeding paragraph implies \[ \Phi^G M \R_{\hat{G}} \cong \left( \bigvee_{W \in I_{R \R(\hat{G})}} S^{\rho \cdot (dim_{\C} W)} \right) \wedge M \R \wedge \left( \prod_{V \in Irr^*(\hat{G})} B \R \right) \] where $I_{R \R(\hat{G})}$ denotes the augmentation ideal of the Real representation ring of $\hat{G}$.
     Forgetting the $C_2$ action recovers the geometric fixed points of $MU$,
     which are given by \[ \Phi^G MU_G \cong \left( \bigvee_{W \in I_{R (G)}} S^{dim_{\R} W} \right) \wedge MU \wedge \left( \prod_{V \in Irr^*(G)} BU \right) \] (where we implicitly use Lemma \ref{lem:RR-RU-identification} to identify Real $\hat{G}$-representations with complex $G$-representations).

     Now it remains to check that the restriction map from $C_2$ to the trivial group induces an isomorphism $\pi_{\rho *}(\Phi^G M \R_{\hat{G}}) \cong \pi_{2*} (\Phi^G MU_G)$. When $\hat{G} = \{e\} \rtimes C_2$, this is the Hu-Kriz isomorphism $\pi_{\rho *}(M \R) \cong \pi_{2 *}(MU)$ \cite{HK01}. When $\hat{G}$ is nontrivial, the result follows from the lemma below.
\end{proof}

\begin{lemma}
The restriction map
\[
res^{C_2}_e:M\R_{*\rho}^{C_2}B\R\to MU_{2*}BU
\]
is an isomorphism. Thus
\[
res^{C_2}_e:M\R_{*\rho}^{C_2}(\underset{I}{\prod} B\R)\to MU_{2*}(\underset{I}{\prod} BU)
\]
is also an isomorphism for $I$ a countable set.
\end{lemma}
\begin{proof}
From \cite[Corollary 5.18]{HHR}, we get an equivalence of $C_2$-spectra
\[
E\wedge M\R\simeq E[\bar{b}_1,\bar{b}_2,\cdots]
\]
for Real-oriented $C_2$-spectra $E$. Take $E=M\R$ and taking homotopy groups we get
\[
M\R_{*\rho}^{C_2}M\R\cong M\R_{*\rho}[\bar{b}_1,\bar{b}_2,\cdots].
\]
Using the Real Thom isomorphism from \cite[Theorem 4.5]{Fujii}, we get
\[
M\R_{*\rho}^{C_2}B\R\cong M\R_{*\rho}[\bar{b}_1,\bar{b}_2,\cdots].
\]
The generators $\bar{b}_i$ restricts to $b_i$ in
\[
MU_{2*}BU\cong MU_{2*}[b_1,b_2,\cdots].
\]
By the Hu-Kriz isomorphism $res^{C_2}_e:M\R_{*\rho}\cong MU_{2*}$ from \cite{HK01}, we deduce the first conclusion.

For the second claim, notice that $M\R^{C_2}_{*\rho}B\R$ is free over $M\R^{C_2}_{*\rho}$. Similarly $MU_*BU$ is free over $MU_*$. The Kunneth formula implies the claim.

\end{proof}

Now we give an alternative proof of Theorem \ref{thm: Hu-Kriz iso} for augmented-tori based on Conjecture \ref{conj:R-is-regular}.

\begin{proof}
    It suffices to show injectivity. Observe that we have the commutative square
    \begin{equation*}
    \begin{tikzcd}
    \mathbf{R}(\mathbb{T}^n) \arrow{r} \arrow{d} & \mathbf{L}(\mathbb{T}^n) \arrow{d} \\
    \Phi^{\mathbb{T}^n} \mathbf{R} \arrow{r} & \Phi^{\mathbb{T}^n} \mathbf{L}  
    \end{tikzcd}    
    \end{equation*}
    By Conjecture \ref{conj:R-is-regular} the left vertical arrow is injective, and by Theorem \ref{thm:res-induces-iso-on-geom-fixed-points} and Proposition \ref{prop:algebraic-gfp-are-same-as-topological-gfp}, the lower horizontal map is injective. Thus the upper horizontal map is injective, as desired.
\end{proof}

\subsection{The Evenness Conjecture implies the Regularity Conjecture}
\label{sec: Evenness Conj implies Regularity Conj}

In this section, we prove that Conjecture \ref{conj:MR-eta-is-even} implies Conjecture \ref{conj:R-is-regular}. The following may be thought of as a Real analogue of \cite[Lemma 5.23]{Hau22}.

\begin{proposition}\label{prop: Evenness Conj implies Regularity Conj}
    Conjecture \ref{conj:MR-eta-is-even} implies Conjecture \ref{conj:R-is-regular}. On the other hand, if Conjecture \ref{conj:R-is-regular} holds, then \[ \pi^G_{*\rho-1}M\R_G=0 \] whenever $G=\widehat{\mathbb{T}}^n$ for some $n$.
\end{proposition}

\begin{proof}
    The Evenness Conjecture implies that the restriction $\mathbf{R} \rightarrow \mathbf{L}$ is an isomorphism. Since $\mathbf{L}$ is regular \cite{Hau22}, so is $\mathbf{R}$.

    On the other hand, the Regularity Conjecture implies the restriction isomorphism for $G=\widehat{\mathbb{T}}^n$. Then in the long exact sequence     \begin{equation}\label{eq:a_sigma long sequence}
    \cdots\to \pi^G_{*\rho+\sigma}M\R_{G}\xrightarrow{a_{\sigma}} \pi^G_{*\rho}M\R_{G} \xrightarrow{res^G_{\mathbb{T}^n}}\pi^{\mathbb{T}^n}_{2*}MU_{\mathbb{T}^n}\to\cdots   
    \end{equation}
    the left term $\pi^{\mathbb{T}^n}_{2*+1}MU_{\mathbb{T}^n}=0$, which implies $\pi^G_{*\rho+\sigma}M\R_G=0$. Using the same arguments as in the proof of Theorem \ref{thm: Hu-Kriz iso}, we deduce $\pi^G_{*\rho-1}M\R_G=0$.
\end{proof}

The authors know of two possible methods one might use to prove either conjecture. Classically, one may deduce evenness of $MU_A^*$ via an Atiyah-Segal completion theorem. One might hope to prove an $RO(C_2)$-graded Atiyah-Segal completion theorem for $M \mathbf{R}_\eta$. On the other hand, there is an alternative argument (adapted from an unpublished argument of Markus Hausmann which gives an alternate proof of the evenness of $MU_A^*$) that requires only the nonexistence of infinitely Euler-divisible elements in $M \mathbf{R}_{\eta}^{\rho *}$.

\begin{theorem}
    If Conjecture \ref{conj:MR-eta-is-even} holds (and hence both Conjectures hold), then the restriction map
    \[
    Res^{\hat{G}}_G : \pi^{\hat{G}}_{*\rho}(M \mathbb{R}_{\hat{G}}) \rightarrow \pi^{G}_*(MU_G) 
    \]
    is an isomorphism for all compact abelian Lie groups $G$. Additionally, the above restriction being an isomorphism for all compact abelian Lie group $G$ is equivalent to Conjecture \ref{conj:MR-eta-is-even}.
\end{theorem}

\begin{proof}
    Assume Conjecture \ref{conj:MR-eta-is-even}, we know that $\mathbf{R}$ is regular by Proposition \ref{prop: Evenness Conj implies Regularity Conj}. The proof of the backwards direction of (1) of \cite[Lemma 5.23]{Hau22} goes through with the following alteration; we use the notation in the proof of \cite[Lemma 5.23]{Hau22}. Let $G=A$ be a compact abelian Lie group. Choose an embedding $i_A:A\to T_A$ of $A$ into a torus. Let $V_k : T_A \rightarrow \mathbb{T}$ be a character of the torus $T_A$ with $B \subset ker(V_k)$. The long exact sequence in $\mathbf{MR}$-cohomology obtained from the cofiber sequence \[ \hat{B}/\widehat{\mathrm{ker}(Res_B^{T_A}(V_k))}_+ \rightarrow S^0 \rightarrow S^{V_k} \] must be viewed as $RO(C_2)$-graded, a priori. However, in degrees $\rho *$, the kernel of multiplication by $e_{Res_B^{T_A}(V_k)}$ is zero by regularity, and the cokernel of the restriction is zero by the evenness conjecture. The resulting collection of short exact sequences implies the isomorphism \[ \pi_{\rho *}^{\hat{B}}( \mathbf{MR} )/(e_{Res_B^{T_A}(V_k)}) \xrightarrow{\cong} \pi_{\rho *}^{\hat{A}}( \mathbf{MR} ) \mathrm{.} \] The remainder of the proof follows from induction and purely algebraic properties of global group laws, as in \cite[Lemma 5.23]{Hau22}.

    On the other hand, if the restriction $Res^{\hat{G}}_G$ is an isomorphism for all compact abelian Lie groups $G$, then using the long exact sequence in (\ref{eq:a_sigma long sequence}) with $\mathbb{T}^n$ replaced by $G$ and $G$ replaced by $\hat{G}$, we deduce Conjecture \ref{conj:MR-eta-is-even} as in the proof of Proposition \ref{prop: Evenness Conj implies Regularity Conj}.
\end{proof}

\bibliographystyle{alpha}
\phantomsection\addcontentsline{toc}{section}{\refname}

\newcommand{\etalchar}[1]{$^{#1}$}

\end{document}